\documentclass[12pt,etds]{amsart}
\usepackage{amsmath}
\usepackage{amssymb}
\usepackage{amsfonts}
\usepackage{amsthm}
\usepackage{enumerate}
\usepackage{tikz}
\usetikzlibrary{matrix}

\usepackage{pb-diagram}

\newtheorem{theorem}{Theorem}[section]
\newtheorem{lemma}[theorem]{Lemma}
\newtheorem{prop}[theorem]{Proposition}
\newtheorem{cor}[theorem]{Corollary}
\theoremstyle{remark}
\newtheorem{definition}[theorem]{Definition}

\newtheorem{remark}[theorem]{Remark}

\def\C{{\mathbb C}}

\def\TT{{\mathbb T}}

\def\B{{\mathcal{B}}}

\def\K{{\mathcal{K}}}

\def\T{{\mathcal{T}}}
\def\I{{\mathcal{I}}}
\def\J{{\mathcal{J}}}
\def\L{{\mathcal{L}}}

\def\M{{\mathcal{M}}}

\newcommand{\clsp}{\overline{\operatorname{span}}}
\newcommand{\lsp}{\operatorname{span}}

\newcommand{\id}{\operatorname{id}}

\newcommand{\piso}{\operatorname{piso}}
\newcommand{\iso}{\operatorname{iso}}

\newcommand{\End}{\operatorname{End}}
\newcommand{\Aut}{\operatorname{Aut}}

\newcommand{\whitesquare}{\hfill $\whitesquare$\newline\vspace{0.4cm}}

\def\newspan{\operatorname{span}}

\numberwithin{equation}{section}

\begin{document}

\title[Nica-Toeplitz algebra]
{The Nica-Toeplitz algebras of dynamical systems over abelian lattice-ordered groups as full corners}

\author[Saeid Zahmatkesh]{Saeid Zahmatkesh}
\address{Department of Mathematics, Faculty of Science, King Mongkut's University of Technology Thonburi, Bangkok 10140, THAILAND}
\email{saeid.zk09@gmail.com, saeid.kom@kmutt.ac.th}



\subjclass[2010]{Primary 46L55}
\keywords{Nica-Toeplitz algebra, lattice-ordered, semigroup, crossed product}

\begin{abstract}
Consider the pair $(G,P)$ consisting of an abelian lattice-ordered discrete group $G$ and its positive cone $P$. Let $\alpha$ be
an action of $P$ by extendible endomorphisms of a $C^*$-algebra $A$. We show that the Nica-Toeplitz
algebra $\mathcal{T}_{\textrm{cov}}(A\times_{\alpha} P)$ is a full corner of a group crossed product $\mathcal{B}\rtimes_{\beta}G$,
where $\mathcal{B}$ is a subalgebra of $\ell^{\infty}(G,A)$ generated by a collection of faithful copies of $A$, and
the action $\beta$ on $\mathcal{B}$ is given by the shift on $\ell^{\infty}(G,A)$. By using this realization, we show that
the ideal $\mathcal{I}$ of $\mathcal{T}_{\textrm{cov}}(A\times_{\alpha} P)$ for which the quotient algebra
$\mathcal{T}_{\textrm{cov}}(A\times_{\alpha} P)/\mathcal{I}$ is the isometric crossed product
$A\times_{\alpha}^{\textrm{iso}} P$ is also a full corner in an ideal $\mathcal{J}\rtimes_{\beta}G$ of $\mathcal{B}\rtimes_{\beta}G$.
\end{abstract}
\maketitle

\section{Introduction}
\label{intro}
Let $P$ be the positive cone of an abelian lattice-ordered discrete group $G$. The identity element of $G$ is denoted by $e$, and $s^{-1}$ denotes the inverse of an element $s\in G$. Note that we have $P^{-1}\cap P=\{e\}$ and $G=P^{-1}P$. For every $s,t\in G$, we let $s\vee t$ and $s\wedge t$ denote
the supremum and infimum of the elements $s$ and $t$, respectively. Suppose that $(A,P,\alpha)$ is a dynamical system consisting of a $C^*$-algebra $A$, and an action $\alpha:P\rightarrow \End (A)$ of $P$ by endomorphisms of $A$ such that
$\alpha_{e}=\id$. Note that since the $C^*$-algebra $A$ is not necessarily unital, we need to assume that each endomorphism $\alpha_{s}$ is extendible, which means that, it extends to a strictly continuous endomorphism $\overline{\alpha}_{s}$ of the multiplier algebra $\M(A)$.
Recall that an endomorphism $\alpha$ of $A$ is extendible if and only if there exists an approximate identity $\{a_{\lambda}\}$ in $A$ and a projection $p\in \M(A)$ such that $\alpha(a_{\lambda})$ converges strictly to $p$ in $\M(A)$. However, the extendibility of $\alpha$ does not necessarily imply $\overline{\alpha}(1_{\M(A)})=1_{\M(A)}$. In \cite{Fowler}, for the system $(A,P,\alpha)$, Fowler defined a covariant representation called the
\emph{Nica-Toeplitz covariant representation} of the system, such that the endomorphisms $\alpha_{s}$ are implemented by partial isometries.
He then showed that there exists a universal $C^*$-algebra $\T_{\textrm{cov}}(A\times_{\alpha} P)$ associated with the system $(A,P,\alpha)$
generated by a universal Nica-Toeplitz covariant representation of the system such that there is a bijection between the Nica-Toeplitz
covariant representations of the system and the nondegenerate representations of $\T_{\textrm{cov}}(A\times_{\alpha} P)$. This universal
algebra is called the \emph{Nica-Toeplitz algebra} or \emph{Nica-Toeplitz crossed product} of the system $(A,P,\alpha)$. We recall that
when the group $G$ is totally ordered and abelian, the algebra $\T_{\textrm{cov}}(A\times_{\alpha} P)$ is the partial-isometric crossed product
$A\times_{\alpha}^{\piso} P$ of the system $(A,P,\alpha)$ introduced and studied in \cite{LR}. Further studies on the structure of the algebra
$A\times_{\alpha}^{\piso} P$ were carried out in \cite{Adji-Abbas}, \cite{AZ}, \cite{AZ2}, \cite{LZ}, and \cite{SZ}. In particular, it
was shown in \cite{SZ} that $A\times_{\alpha}^{\piso} P$ is a full corner in a usual crossed product by a group. This is the main
inspiration of the present work, where by following the framework of \cite{SZ}, we generalize this corner realization of $\T_{\textrm{cov}}(A\times_{\alpha} P)$ to more general groups, namely, (abelian) lattice-ordered groups. However, compared to
the totally ordered case, the discussions here are more complicated which involve quite huge computations. We think that
such efforts are very useful on understanding the structure of $C^*$-algebras constructed out of semigroup dynamical systems, on which
we could import many information from the well-established theory of the usual crossed products by groups (for example, see \cite{LZ}).
This construction has been used recently in \cite{Hum} by Humeniuk in the $C^*$-envelope programme of Davidson-Fuller-Kakariadis in \cite{DFK}.

Following the idea of \cite{SZ}, a subalgebra $\B$ of the algebra $\ell^{\infty}(G,A)$ of norm bounded $A$-valued functions of $G$ will be defined. Then the shift on $\ell^{\infty}(G,A)$ gives an action $\beta$ of $G$ on $\B$ by automorphisms. Let $\B\rtimes_{\beta} G$ be the associated group crossed product of $\B$ by $G$. Next, a Nica-Toeplitz covariant representation of $(A,P,\alpha)$ in the multiplier algebra of $\B\rtimes_{\beta} G$ will be constructed, and we show that the corresponding homomorphism of the Nica-Toeplitz algebra is an isomorphism of $\T_{\textrm{cov}}(A\times_{\alpha} P)$ onto a full corner of $\B\rtimes_{\beta} G$. We then apply this realization to show that the
kernel of the natural surjective homomorphism $\Omega:(\T_{\textrm{cov}}(A\times_{\alpha} P),i_{A},i_{P})\rightarrow (A\times_{\alpha}^{\iso} P,\mu_{A},\mu_{P})$ induced by the canonical isometric covariant pair
$(\mu_{A},\mu_{P})$ of $(A,P,\alpha)$ is also a full corner in an ideal $\mathcal{J}\rtimes_{\beta}G$ of $\mathcal{B}\rtimes_{\beta}G$.
Moreover, we will see that $\ker \Omega$ is an essential ideal.

We begin with a preliminary section containing a summary on the Nica-Toeplitz algebra and the theory of the isometric crossed
products. In section \ref{sec:alg B} we introduce a subalgebra $\B$ of
$\ell^{\infty}(G,A)$, and an essential ideal $\J$ of $\B$. In section \ref{sec:full piso} we consider the usual crossed product
$\B\rtimes_{\beta} G$ by the group $G$, where the action $\beta$ is given by the shift on $\ell^{\infty}(G,A)$. Then
a Nica-Toeplitz covariant representation of $(A,P,\alpha)$ in $\M(\B\rtimes_{\beta} G)$ will be constructed, from which
we get an isomorphism $\Psi$ of the Nica-Toeplitz algebra $\T_{\textrm{cov}}(A\times_{\alpha} P)$ onto a full corner of $\B\rtimes_{\beta} G$.
Also, as $\J$ is a $\beta$-invariant essential ideal of $\B$, we identify the ideal $\ker \Omega$ with a full corner in the ideal $\mathcal{J}\rtimes_{\beta}G$ of $\mathcal{B}\rtimes_{\beta}G$ via the isomorphism $\Psi$.
Finally, in section \ref{sec:auto} we show that when the action of $P$ is given by automorphisms the Nica-Toeplitz algebra
$\T_{\textrm{cov}}(A\times_{\alpha} P)$ is a full corner in the usual crossed product $(B_{G}\otimes A)\rtimes G$ by group.

\section{Preliminaries}
\label{sec:pre}

\subsection{Morita equivalence and full corner}
\label{Morita}
The $C^*$-algebras $A$ and $B$ are called \emph{Morita equivalent} if there is an $A$--$B$-imprimitivity bimodule $X$. If $p$ is a projection in the multiplier algebra $\M(A)$ of $A$,
then the  $C^*$-subalgebra $pAp$ of $A$ is called a \emph{corner} in $A$. We
say a corner $pAp$ is \emph{full} if $\overline{ApA}:=\clsp\{apb: a,b\in A\}$ is $A$. Any full corner of $A$ is Morita equivalent to $A$ via the imprimitivity bimodule $Ap$
(see more in \cite{Lance} or \cite{RW}).

\subsection{Nica-Toeplitz algebra}
\label{NT-Alg}

A \emph{partial-isometric representation} of $P$ on a Hilbert space $H$ is a map $V:P\rightarrow B(H)$ such that each
$V_{x}:=V(x)$ is a partial isometry, and $V_{x}V_{y}=V_{xy}$ for all $x,y\in P$.

A \emph{Toeplitz covariant representation} of $(A,P,\alpha)$ on a Hilbert space $H$ is a pair $(\pi,V)$ consisting of
a nondegenerate representation $\pi:A\rightarrow B(H)$ and a partial-isometric representation $V:P\rightarrow B(H)$ of $P$ such that
\begin{align}
\label{cov1}
\pi(\alpha_{x}(a))=V_{x}\pi(a) V_{x}^{*}\ \ \textrm{and}\ \ V_{x}^{*}V_{x} \pi(a)=\pi(a) V_{x}^{*}V_{x}
\end{align}
for all $a\in A$ and $x\in P$. A \emph{Nica-Toeplitz covariant representation} of $(A,P,\alpha)$
on a Hilbert space $H$ is a Toeplitz covariant representation $(\pi,V)$ such that
\begin{align}
\label{cov5}
V_{x}^{*}V_{x}V_{y}^{*}V_{y}=V_{x\vee y}^{*}V_{x\vee y}
\end{align}
for all $x,y\in P$. The equation (\ref{cov5}) is called the \emph{Nica covariance}.

Note that every Nica-Toeplitz covariant pair $(\pi,V)$ extends to a Nica-Toeplitz covariant representation $(\overline{\pi},V)$ of
$(M(A),P,\overline{\alpha})$, and (\ref{cov1}) is equivalent to
\begin{align}
\label{cov2}
\pi(\alpha_{x}(a))V_{x}=V_{x}\pi(a)\ \ \textrm{and}\ \ V_{x}V_{x}^{*}=\overline{\pi}(\overline{\alpha}_{x}(1))
\end{align}
for $a\in A$ and $x\in P$.

\begin{definition}
\label{NT-CP-df}
A \emph{Nica-Toeplitz crossed product} of $(A,P,\alpha)$ is a triple $(B,i_{A},i_{P})$ consisting of a $C^*$-algebra $B$, a nondegenerate homomorphism $i_{A}:A\rightarrow B$, and a map
$i_{P}:P\rightarrow \M(B)$ such that:
\begin{itemize}
\item[(i)] if $\Lambda$ is a nondegenerate representation of $B$, then the pair $(\Lambda\circ i_{A},\overline{\Lambda}\circ i_{P})$ is a Nica-Toeplitz covariant representation of $(A,P,\alpha)$;
\item[(ii)] for every Nica-Toeplitz covariant representation $(\pi,V)$ of $(A,P,\alpha)$ on a Hilbert space $H$,
there exists a nondegenerate representation $\pi\times V: B\rightarrow B(H)$ such that $(\pi\times V) \circ i_{A}=\pi$ and $(\overline{\pi\times V}) \circ i_{P}=V$; and
\item[(iii)] $B$ is generated by $\{i_{A}(a)i_{P}(x): a\in A, x\in P\}$.
\end{itemize}
Fowler in \cite{Fowler} showed that the Nica-Toeplitz crossed product of $(A,P,\alpha)$ exists, and it is unique up to isomorphism (see in particular \cite[Proposition 9.2]{Fowler}). He denotes this algebra by
$\T_{\textrm{cov}}(A\times_{\alpha} P)$, which is also called the \emph{Nica-Toeplitz algebra}.
\end{definition}

\begin{remark}
\label{Rem-cov3}
Note that in the definition \ref{NT-CP-df}, as the algebra $B$ can be embedded in some algebra $B(H)$ by a faithful nondegenerate representation, (i) is indeed
equivalent to the following statement:
\begin{itemize}
\item[(1)] the pair $(i_{A}, i_{P})$ is a Nica-Toeplitz covariant representation of $(A,P,\alpha)$ in $B$.
\end{itemize}
This means that the pair $(i_{A}, i_{P})$ consists of a nondegenerate homomorphism $i_{A}:A\rightarrow B$ and a partial-isometric representation
$i_{P}:P\rightarrow \M(B)$ which satisfy the covariance equations
\begin{align}
\label{cov3}
i_{A}(\alpha_{x}(a))=i_{P}(x)i_{A}(a) i_{P}(x)^{*}\ \ \textrm{and}\ \ i_{P}(x)^{*}i_{P}(x) i_{A}(a)=i_{A}(a) i_{P}(x)^{*}i_{P}(x),
\end{align}
and the Nica covariance equation
\begin{align}
\label{cov4}
i_{P}(x)^{*}i_{P}(x)i_{P}(y)^{*}i_{P}(y)=i_{P}(x\vee y)^{*}i_{P}(x\vee y)
\end{align}
for all $a\in A$ and $x,y\in P$. So, it follows that we actually have
\begin{align}
\label{span-B}
B=\overline{\newspan}\{i_{P}(x)^{*} i_{A}(a) i_{P}(y) : x,y \in P, a\in A\}.
\end{align}
To see this, we only have to show that the right hand side of $(\ref{span-B})$ is closed under multiplication. To do so, we apply
the Nica covariance equation (\ref{cov4}) to show that each product
\begin{align}
\label{prod-1}
[i_{P}(x)^{*} i_{A}(a) i_{P}(y)][i_{P}(s)^{*} i_{A}(b) i_{P}(t)]
\end{align}
of spanning elements is an element of the same form. We have
\begin{eqnarray*}
\begin{array}{l}
[i_{P}(x)^{*} i_{A}(a) i_{P}(y)][i_{P}(s)^{*} i_{A}(b) i_{P}(t)]\\
=i_{P}(x)^{*} i_{A}(a) i_{P}(y)[i_{P}(y)^{*}i_{P}(y)i_{P}(s)^{*}i_{P}(s)]i_{P}(s)^{*} i_{A}(b) i_{P}(t)\\
=i_{P}(x)^{*} i_{A}(a) i_{P}(y)i_{P}(z)^{*}i_{P}(z)i_{P}(s)^{*} i_{A}(b) i_{P}(t)\ \ \ \ \ [z:=y\vee s]\\
=i_{P}(x)^{*} i_{A}(a) [i_{P}(y)i_{P}(y)^{*}]i_{P}(zy^{-1})^{*}i_{P}(zs^{-1})[i_{P}(s)i_{P}(s)^{*}] i_{A}(b) i_{P}(t)\\
=i_{P}(x)^{*} i_{A}(a) \overline{i_{A}}(\overline{\alpha}_{y}(1)) i_{P}(zy^{-1})^{*}i_{P}(zs^{-1}) \overline{i_{A}}(\overline{\alpha}_{s}(1))
i_{A}(b) i_{P}(t)\\
=i_{P}(x)^{*} i_{A}(a\overline{\alpha}_{y}(1)) i_{P}(zy^{-1})^{*}i_{P}(zs^{-1}) i_{A}(\overline{\alpha}_{s}(1)b) i_{P}(t)\\
=i_{P}(x)^{*}i_{P}(zy^{-1})^{*} i_{A}(\alpha_{zy^{-1}}(c))  i_{A}(\alpha_{zs^{-1}}(d)) i_{P}(zs^{-1}) i_{P}(t)\\
=i_{P}\big((zy^{-1})x\big)^{*} i_{A}\big(\alpha_{zy^{-1}}(c)\alpha_{zs^{-1}}(d)\big) i_{P}\big((zs^{-1})t\big),\\
\end{array}
\end{eqnarray*}
which belongs to the right hand side of $(\ref{span-B})$, where $c=a\overline{\alpha}_{y}(1)$ and $d=\overline{\alpha}_{s}(1)b$.
Thus, $(\ref{span-B})$ indeed holds.
\end{remark}

We recall that by \cite[Theorem 9.3]{Fowler}, a Nica-Toeplitz covariant representation $(\pi, V)$ of $(A,P,\alpha)$ on $H$ induces a faithful representation $\pi\times V$ of $\T_{\textrm{cov}}(A\times_{\alpha} P)$ if and only if for every
finite subset $F=\{x_{1}, x_{2}, ..., x_{n}\}$ of $P\backslash \{e\}$, $\pi$ is faithful on the range of
$$\prod_{i=1}^{n}(1-V_{x_{i}}^{*}V_{x_{i}})=0.$$

\subsection{Isometric crossed products}
\label{Iso-CP}
An \emph{isometric covariant representation} of $(A,P,\alpha)$ on a Hilbert space $H$ is a pair $(\pi,V)$ consisting of
a nondegenerate representation $\pi:A\rightarrow B(H)$ and an isometric representation $V:P\rightarrow B(H)$ of $P$ such that
\begin{align}
\label{iso-cov}
\pi(\alpha_{x}(a))=V_{x}\pi(a) V_{x}^{*}
\end{align}
for all $a\in A$ and $x\in P$.

\begin{definition}
\label{Iso-CP-df}
An \emph{isometric crossed product} of $(A,P,\alpha)$ is a triple $(C,\mu_{A},\mu_{P})$ consisting of a $C^*$-algebra $C$, a nondegenerate homomorphism $\mu_{A}:A\rightarrow C$, and an isometric representation $\mu_{P}:P\rightarrow \M(C)$ such that:
\begin{itemize}
\item[(i)] $\mu_{A}(\alpha_{x}(a))=\mu_{P}(x)\mu_{A}(a) \mu_{P}(x)^{*}$ for all $a\in A$ and $x\in P$;
\item[(ii)] for every isometric covariant representation $(\pi,V)$ of $(A,P,\alpha)$ on a Hilbert space $H$,
there exists a nondegenerate representation $\pi\times V: C\rightarrow B(H)$ such that $(\pi\times V) \circ \mu_{A}=\pi$ and
$(\overline{\pi\times V}) \circ \mu_{P}=V$; and
\item[(iii)] $C$ is generated by $\{\mu_{A}(a)\mu_{P}(x): a\in A, x\in P\}$, indeed we have
$$C=\overline{\newspan}\{\mu_{P}(x)^{*} \mu_{A}(a) \mu_{P}(y) : x,y \in P, a\in A\}.$$
\end{itemize}
Note that the isometric crossed product of the system $(A,P,\alpha)$ exists if the system admits a nontrivial
(isometric) covariant representation, and it is unique up to isomorphism. Thus, the isometric crossed product of the system
$(A,P,\alpha)$ is denoted by $A\times_{\alpha}^{\iso} P$. We refer readers to \cite{ALNR, Fowler, Marcelo, LacaR, Larsen, Stacey} for more
on isometric crossed products.
\end{definition}
Consider the dynamical system $(A,P,\alpha)$ in which the action
$\alpha$ of $P$ is given by extendible endomorphisms of $A$. Let $(\T_{\textrm{cov}}(A\times_{\alpha} P),i_{A},i_{P})$ and $(A\times_{\alpha}^{\iso} P,\mu_{A},\mu_{P})$ be the Nica-Toeplitz algebra and the isometric crossed product associated with the system, respectively.
One can see that the pair $(\mu_{A},\mu_{P})$ is a Nica-Toeplitz covariant representation of $(A,P,\alpha)$ in the $C^*$-algebra $A\times_{\alpha}^{\iso} P$. Thus, there exists a nondegenerate homomorphism
$$\Omega:(\T_{\textrm{cov}}(A\times_{\alpha} P),i_{A},i_{P})\rightarrow (A\times_{\alpha}^{\iso} P,\mu_{A},\mu_{P})$$ such that
$$\Omega(i_{P}(x)^{*}i_{A}(a)i_{P}(y))=\mu_{P}(x)^{*}\mu_{A}(a)\mu_{P}(y)$$ for all $a\in A$ and $x,y\in P$. So, it follows that $\Omega$ is surjective, and hence, we have the following short exact sequence:
\begin{align}
\label{exseq1}
0 \longrightarrow \ker \Omega \stackrel{}{\longrightarrow} \T_{\textrm{cov}}(A\times_{\alpha} P) \stackrel{\Omega}{\longrightarrow} A\times_{\alpha}^{\iso} P \longrightarrow 0,
\end{align}

Next, we want to identify spanning elements for the ideal $\ker \Omega$. To do so, first, see that, for every $r,s\in P$, by the covariance equation of $(i_{A},i_{P})$, we have
\begin{eqnarray*}
\begin{array}{rcl}
i_{P}(r)[1-i_{P}(s)^{*}i_{P}(s)]&=&i_{P}(r)-i_{P}(r)i_{P}(s)^{*}i_{P}(s)\\
&=&i_{P}(r)-i_{P}(r)[i_{P}(r)^{*}i_{P}(r)i_{P}(s)^{*}i_{P}(s)]\\
&=&i_{P}(r)-i_{P}(r)i_{P}(r\vee s)^{*}i_{P}(r\vee s)\\
&=&i_{P}(r)-i_{P}(r)i_{P}(r)^{*}i_{P}((r\vee s)r^{-1})^{*}i_{P}((r\vee s)r^{-1})i_{P}(r)\\
&=&i_{P}(r)-\overline{i_{A}}(\overline{\alpha}_{r}(1))i_{P}((r\vee s)r^{-1})^{*}i_{P}((r\vee s)r^{-1})i_{P}(r)\\
&=&i_{P}(r)-i_{P}((r\vee s)r^{-1})^{*}i_{P}((r\vee s)r^{-1})\overline{i_{A}}(\overline{\alpha}_{r}(1))i_{P}(r)\\
&=&i_{P}(r)-i_{P}((r\vee s)r^{-1})^{*}i_{P}((r\vee s)r^{-1})i_{P}(r)\overline{i_{A}}(1)\\
&=&[1-i_{P}((r\vee s)r^{-1})^{*}i_{P}((r\vee s)r^{-1})]i_{P}(r).
\end{array}
\end{eqnarray*}
Therefore,
\begin{align}
\label{eq1}
i_{P}(r)[1-i_{P}(s)^{*}i_{P}(s)]=[1-i_{P}((r\vee s)r^{-1})^{*}i_{P}((r\vee s)r^{-1})]i_{P}(r)
\end{align}
for every $r,s\in P$. This equation will be applied in the following proposition regarding the identifying spanning elements for the ideal $\ker \Omega$.

\begin{prop}
\label{ess ideal of NT}
Let $$\I:=\clsp\{i_{P}(x)^{*}i_{A}(a)(1-i_{P}(s)^{*}i_{P}(s))i_{P}(y): a\in A, x,y,s\in P\}.$$ Then $\I$ is an ideal of $(\T_{\textrm{cov}}(A\times_{\alpha} P),i_{A},i_{P})$, and $\ker \Omega=\I$.
\end{prop}

\begin{proof}
To see that $\I$ is an ideal of $\T_{\textrm{cov}}(A\times_{\alpha} P)$, it suffices to show that $\I$ is a $^{*}$-algebra, and
$i_{P}(t)^{*}\I$, $i_{P}(t)\I$, and $i_{A}(b)\I$ are all contained in $\I$ for every $t\in P$ and $b\in A$
(by only computing on the spanning elements of $\I$). As
$$\big[i_{P}(x)^{*}i_{A}(a)(1-i_{P}(s)^{*}i_{P}(s))i_{P}(y)\big]^{*}=i_{P}(y)^{*}i_{A}(a^{*})(1-i_{P}(s)^{*}i_{P}(s))i_{P}(x)\in \I,$$
$\I$ is indeed a $^{*}$-algebra.

Now, $i_{P}(t)^{*}\I\subset \I$ is trivial. To see the second one, $i_{P}(t)\I\subset \I$, first note that,
by applying the covariance equation of $(i_{A},i_{P})$, we have
\begin{eqnarray*}
\begin{array}{rcl}
i_{P}(t)i_{P}(x)^{*}&=&i_{P}(t)[i_{P}(t)^{*}i_{P}(t)i_{P}(x)^{*}i_{P}(x)]i_{P}(x)^{*}\\
&=&i_{P}(t)i_{P}(t\vee x)^{*}i_{P}(t\vee x)i_{P}(x)^{*}\\
&=&i_{P}(t)i_{P}(t)^{*}i_{P}((t\vee x)t^{-1})^{*}i_{P}((t\vee x)x^{-1})i_{P}(x)i_{P}(x)^{*}\\
&=&\overline{i_{A}}(\overline{\alpha}_{t}(1))i_{P}((t\vee x)t^{-1})^{*}i_{P}((t\vee x)x^{-1})\overline{i_{A}}(\overline{\alpha}_{x}(1))\\
&=&i_{P}((t\vee x)t^{-1})^{*}\overline{i_{A}}(\overline{\alpha}_{(t\vee x)t^{-1}}(\overline{\alpha}_{t}(1)))
\overline{i_{A}}(\overline{\alpha}_{(t\vee x)x^{-1}}(\overline{\alpha}_{x}(1)))i_{P}((t\vee x)x^{-1})\\
&=&i_{P}((t\vee x)t^{-1})^{*}\overline{i_{A}}(\overline{\alpha}_{(t\vee x)}(1))
\overline{i_{A}}(\overline{\alpha}_{(t\vee x)}(1))i_{P}((t\vee x)x^{-1})\\
&=&i_{P}((t\vee x)t^{-1})^{*}\overline{i_{A}}(\overline{\alpha}_{(t\vee x)}(1))i_{P}((t\vee x)x^{-1}).
\end{array}
\end{eqnarray*}
Therefore,
\begin{eqnarray*}
\begin{array}{l}
i_{P}(t)[i_{P}(x)^{*}i_{A}(a)(1-i_{P}(s)^{*}i_{P}(s))i_{P}(y)]\\
=[i_{P}(t)i_{P}(x)^{*}]i_{A}(a)(1-i_{P}(s)^{*}i_{P}(s))i_{P}(y)\\
=i_{P}((t\vee x)t^{-1})^{*}\overline{i_{A}}(\overline{\alpha}_{(t\vee x)}(1))[i_{P}((t\vee x)x^{-1})i_{A}(a)](1-i_{P}(s)^{*}i_{P}(s))i_{P}(y)\\
=i_{P}((t\vee x)t^{-1})^{*}\overline{i_{A}}(\overline{\alpha}_{(t\vee x)}(1))i_{A}(\alpha_{(t\vee x)x^{-1}}(a))
i_{P}((t\vee x)x^{-1})(1-i_{P}(s)^{*}i_{P}(s))i_{P}(y)\\
=i_{P}((t\vee x)t^{-1})^{*}i_{A}(c)[i_{P}(r)(1-i_{P}(s)^{*}i_{P}(s))]i_{P}(y),
\end{array}
\end{eqnarray*}
where $c=\overline{\alpha}_{(t\vee x)}(1)\alpha_{(t\vee x)x^{-1}}(a)\in A$ and $r=(t\vee x)x^{-1}\in P$. Then, in the bottom line,
for $i_{P}(r)(1-i_{P}(s)^{*}i_{P}(s))$, we apply (\ref{eq1}), which gives us
\begin{eqnarray*}
\begin{array}{l}
i_{P}(t)[i_{P}(x)^{*}i_{A}(a)(1-i_{P}(s)^{*}i_{P}(s))i_{P}(y)]\\
=i_{P}((t\vee x)t^{-1})^{*}i_{A}(c)[1-i_{P}((r\vee s)r^{-1})^{*}i_{P}((r\vee s)r^{-1})]i_{P}(r)i_{P}(y)\\
=i_{P}((t\vee x)t^{-1})^{*}i_{A}(c)[1-i_{P}((r\vee s)r^{-1})^{*}i_{P}((r\vee s)r^{-1})]i_{P}(ry)
\end{array}
\end{eqnarray*}
which belongs to $\I$. To see the last one, $i_{A}(b)\I\subset \I$, again by using the covariance equation of $(i_{A},i_{P})$, we see that
\begin{eqnarray*}
\begin{array}{l}
i_{A}(b)[i_{P}(x)^{*}i_{A}(a)(1-i_{P}(s)^{*}i_{P}(s))i_{P}(y)]\\
=[i_{A}(b)i_{P}(x)^{*}]i_{A}(a)(1-i_{P}(s)^{*}i_{P}(s))i_{P}(y)\\
=i_{P}(x)^{*}i_{A}(\alpha_{x}(b))i_{A}(a)(1-i_{P}(s)^{*}i_{P}(s))i_{P}(y)\\
=i_{P}(x)^{*}i_{A}(\alpha_{x}(b)a)(1-i_{P}(s)^{*}i_{P}(s))i_{P}(y)\in \I.
\end{array}
\end{eqnarray*}
Thus, $\I$ is an ideal of $\T_{\textrm{cov}}(A\times_{\alpha} P)$.

Now, we show that $\ker \Omega=\I$. The inclusion $\I \subset \ker \Omega$ follows immediately as $\overline{\Omega}(1-i_{P}(s)^{*}i_{P}(s))=1-\mu_{P}(s)^{*}\mu_{P}(s)=0$ for every $s\in P$. For the other inclusion, take a nondegenerate
representation $\sigma$ of $\T_{\textrm{cov}}(A\times_{\alpha} P)$ on some Hilbert space $H$ such that $\ker \sigma=\I$.
Then, $(\pi,V):=(\sigma\circ i_{A},\overline{\sigma}\circ i_{P})$ is a Nica-Toeplitz covariant representation of $(A,P,\alpha)$ on $H$. However,
each $V_{s}$ is actually an isometry. To see this, take any approximate identity $\{a_{\lambda}\}$ in $A$. Then, we have
$$0=\sigma(i_{A}(a_{\lambda})(1-i_{P}(s)^{*}i_{P}(s)))=\pi(a_{\lambda})(1-V_{s}^{*}V_{s})$$ for each $\lambda$. Since $\pi$ is nondegenerate, it
follows that $\pi(a_{\lambda})(1-V_{s}^{*}V_{s})$ converges strongly to $(1-V_{s}^{*}V_{s})$, and hence, we must have
$1-V_{s}^{*}V_{s}=0$. This implies that the pair $(\pi,V)$ is indeed a covariant isometric representation of
$(A,P,\alpha)$ on $H$. Therefore, there is a nondegenerate representation $\varphi$ of the isometric crossed product
$(A\times_{\alpha}^{\iso} P,\mu_{A},\mu_{P})$ on $H$, such that $\varphi(\mu_{A}(a))=\pi(a)=\sigma(i_{A}(a))$ and $\overline{\varphi}(\mu_{P}(x))=V_{x}=\overline{\sigma}(i_{P}(x))$ for all $a\in A$ and $x\in P$. This implies that
$\varphi \circ \Omega=\sigma$, from which, we conclude that $\ker \Omega \subset \ker \sigma$.
\end{proof}

\section{The $C^*$-algebra $\B$ and its ideal $\J$}
\label{sec:alg B}
Let $(G,P)$ be an abelian lattice-ordered group, and $(A,P,\alpha)$ a dynamical system in which $\alpha$ is an action of $P$ by extendible endomorphisms of a $C^*$-algebra $A$.
Consider the algebra $\ell^{\infty}(G,A)$ of all norm bounded $A$-valued functions of $G$. For every $s\in G$, we define a map $\phi_{s}:A\rightarrow\ell^{\infty}(G,A)$ by
\[
\phi_{s}(a)(x)=
   \begin{cases}
      \alpha_{xs^{-1}}(a) &\textrm{if}\empty\ \text{$s\leq x$}\\
      0 &\textrm{otherwise}.
   \end{cases}
\]
It is not difficult to see that each map $\phi_{s}$ is actually an injective $*$-homomorphism (embedding). Now, let $\B$ be the $C^*$-subalgebra of $\ell^{\infty}(G,A)$ generated by $\{\phi_{s}(a):s\in G, a\in A\}$.
Note that, since $\phi_{s}(a)^{*}=\phi_{s}(a^{*})$, and
\begin{align}
\label{suprem}
\phi_{s}(a)\phi_{t}(b)=\phi_{s\vee t}(\alpha_{(s\vee t)s^{-1}}(a)\alpha_{(s\vee t)t^{-1}}(b)),
\end{align}
we actually have $$\B=\clsp\{\phi_{s}(a):s\in G, a\in A\}.$$

Moreover, the elements of $\B$ satisfy the following property:
\begin{lemma}
\label{B property}
Let $\xi\in\B$. Then, for any $\varepsilon>0$, there are $y,z\in G$ such that if $x<y$, then $\|\xi(x)\|<\varepsilon$, and if $x\geq z$, then
$\|\xi(x)-\alpha_{xz^{-1}}(\xi(z))\|<\varepsilon$.
\end{lemma}

\begin{proof}
For any $\varepsilon>0$, there is a finite sum $\sum_{i=0}^{n}\phi_{y_{i}}(a_{i})$ such that
\begin{align}
\label{eq9}
\|\xi-\sum_{i=0}^{n}\phi_{y_{i}}(a_{i})\|<\varepsilon/2.
\end{align}
Take $y=y_{0}\wedge y_{1}\wedge...\wedge y_{n}$. Then, for every $x\in G$, we have
$$\|(\xi-\sum_{i=0}^{n}\phi_{y_{i}}(a_{i}))(x)\|\leq\|\xi-\sum_{i=0}^{n}\phi_{y_{i}}(a_{i})\|<\varepsilon/2<\varepsilon,$$
and therefore, since
$$\|(\xi-\sum_{i=0}^{n}\phi_{y_{i}}(a_{i}))(x)\|=\|\xi(x)-\sum_{i=0}^{n}\phi_{y_{i}}(a_{i})(x)\|,$$ it follows that
$$\|\xi(x)-\sum_{i=0}^{n}\phi_{y_{i}}(a_{i})(x)\|<\varepsilon.$$
Now, if $x<y$, then $\phi_{y_{i}}(a_{i})(x)=0$ for each $0\leq i\leq n$. So, we have $\|\xi(x)\|<\varepsilon$ for every $x<y$.

Next, take $z=y_{0}\vee y_{1}\vee...\vee y_{n}$, and for convenience, let $\xi_{n}=\sum_{i=0}^{n}\phi_{y_{i}}(a_{i})$. Since $\|\xi-\xi_{n}\|<\varepsilon/2$ by (\ref{eq9}), for every $x\geq z$, we get
\begin{eqnarray*}
\begin{array}{l}
\|\xi(x)-\alpha_{xz^{-1}}(\xi(z))\|\\
=\|(\xi-\xi_{n}+\xi_{n})(x)-\alpha_{xz^{-1}}((\xi-\xi_{n}+\xi_{n})(z))\|\\
=\|(\xi-\xi_{n})(x)+\xi_{n}(x)-\alpha_{xz^{-1}}((\xi-\xi_{n})(z)+\xi_{n}(z))\|\\
=\|(\xi-\xi_{n})(x)+\xi_{n}(x)-\alpha_{xz^{-1}}((\xi-\xi_{n})(z))-\alpha_{xz^{-1}}(\xi_{n}(z))\|\\
\leq\|(\xi-\xi_{n})(x)\|+\|-\alpha_{xz^{-1}}((\xi-\xi_{n})(z))\|+\|\xi_{n}(x)-\alpha_{xz^{-1}}(\xi_{n}(z))\|\\
\leq\|\xi-\xi_{n}\|+\|(\xi-\xi_{n})(z)\|+\|\xi_{n}(x)-\alpha_{xz^{-1}}(\xi_{n}(z))\|\\
\leq\|\xi-\xi_{n}\|+\|\xi-\xi_{n}\|+\|\xi_{n}(x)-\alpha_{xz^{-1}}(\xi_{n}(z))\|\\
<\varepsilon/2+\varepsilon/2+\|\xi_{n}(x)-\alpha_{xz^{-1}}(\xi_{n}(z))\|=\varepsilon+\|\xi_{n}(x)-\alpha_{xz^{-1}}(\xi_{n}(z))\|.
\end{array}
\end{eqnarray*}
Therefore, we have
$$\|\xi(x)-\alpha_{xz^{-1}}(\xi(z))\|<\varepsilon+\|\xi_{n}(x)-\alpha_{xz^{-1}}(\xi_{n}(z))\|.$$
However, since $x\geq z$, the following calculation
\begin{eqnarray*}
\begin{array}{rcl}
\xi_{n}(x)-\alpha_{xz^{-1}}(\xi_{n}(z))&=&\sum_{i=0}^{n}\phi_{y_{i}}(a_{i})(x)-\alpha_{xz^{-1}}(\sum_{i=0}^{n}\phi_{y_{i}}(a_{i})(z))\\
&=&\sum_{i=0}^{n}\alpha_{xy_{i}^{-1}}(a_{i})-\alpha_{xz^{-1}}(\sum_{i=0}^{n}\alpha_{zy_{i}^{-1}}(a_{i}))\\
&=&\sum_{i=0}^{n}\alpha_{xy_{i}^{-1}}(a_{i})-\sum_{i=0}^{n}\alpha_{xz^{-1}}(\alpha_{zy_{i}^{-1}}(a_{i}))\\
&=&\sum_{i=0}^{n}\alpha_{xy_{i}^{-1}}(a_{i})-\sum_{i=0}^{n}\alpha_{xy_{i}^{-1}}(a_{i})=0
\end{array}
\end{eqnarray*}
shows that actually $\|\xi_{n}(x)-\alpha_{xz^{-1}}\xi_{n}(z))\|=0$.
It thus follows that $\|\xi(x)-\alpha_{xz^{-1}}(\xi(z))\|<\varepsilon$ for every $x\geq z$.
\end{proof}

\begin{lemma}
\label{mux}
Each homomorphism $\phi_{s}:A\rightarrow\mathcal{B}$ extends to a strictly continuous homomorphism $\overline{\phi}_{s}:\M(A)\rightarrow\M(\B)$ of multiplier algebras.
\end{lemma}

\begin{proof}
Let  $\{a_{\lambda}\}$ be an approximate identity in $A$. We show that there exists a projection $p_{s}\in\M(\B)$
such that $\phi_{s}(a_{\lambda})\rightarrow p_{s}$ strictly in $\M(\B)$. It suffices to see that $\phi_{s}(a_{\lambda})\phi_{t}(a)\rightarrow p_{s}\phi_{t}(a)$ and $\phi_{t}(a)\phi_{s}(a_{\lambda})\rightarrow \phi_{t}(a)p_{s}$
in the norm topology of $\B$ for every $a\in A$ and $t\in G$. Consider the algebra $\ell^{\infty}(G,\M(A))$ which contains $\ell^{\infty}(G,A)$ as an essential ideal.
Then, similar to $\B$, define $\overline{\B}$ to be the $C^*$-subalgebra of $\ell^{\infty}(G,\M(A))$ spanned by $\{\chi_{s}(m): s\in G, m\in \M(A)\}$, where
$\chi_{s}:\M(A)\rightarrow\ell^{\infty}(G,\M(A))$ is a map defined by
\[
\chi_{s}(m)(x)=
   \begin{cases}
      \overline{\alpha}_{xs^{-1}}(m) &\textrm{if}\empty\ \text{$s\leq x$}\\
      0 &\textrm{otherwise}.
   \end{cases}
\]
Each $\chi_{s}$ is then an embedding such that $\chi_{s}|_{A}=\phi_{s}$. Now, since $\B$ sits in $\overline{\B}$ as an essential ideal, $\overline{\B}$ sits in $\M(\B)$ as a $C^*$-subalgebra.
Let $p_{s}=\chi_{s}(1)$ for every $s\in G$, which is a projection in $\M(\B)$. Then, by (\ref{suprem}), we have
$$\phi_{s}(a_{\lambda})\phi_{t}(a)=\phi_{s\vee t}(\alpha_{(s\vee t)s^{-1}}(a_{\lambda})\alpha_{(s\vee t)t^{-1}}(a)),$$ which is convergent to
$$\phi_{s\vee t}(\overline{\alpha}_{(s\vee t)s^{-1}}(1)\alpha_{(s\vee t)t^{-1}}(a))$$ in the norm topology of $\B$ (This is due to the facts that each $\alpha_{x}$ is extendible, and each $\phi_{s}$ is an isometry).
On the other hand, again by a similar equation to (\ref{suprem}) for the spanning elements $\chi_{s}$ of $\overline{\B}$,
\begin{eqnarray*}
\begin{array}{rcl}
p_{s}\phi_{t}(a)=\chi_{s}(1)\phi_{t}(a)&=&\chi_{s}(1)\chi_{t}(a)\\
&=&\chi_{s\vee t}(\overline{\alpha}_{(s\vee t)s^{-1}}(1)\overline{\alpha}_{(s\vee t)t^{-1}}(a))\\
&=&\chi_{s\vee t}(\overline{\alpha}_{(s\vee t)s^{-1}}(1)\alpha_{(s\vee t)t^{-1}}(a))\\
&=&\phi_{s\vee t}(\overline{\alpha}_{(s\vee t)s^{-1}}(1)\alpha_{(s\vee t)t^{-1}}(a))
\end{array}
\end{eqnarray*}
Thus, it follows that $\phi_{s}(a_{\lambda})\phi_{t}(a)$ is indeed convergent to $p_{s}\phi_{t}(a)$ in $\B$.
We also have $\phi_{t}(a)\phi_{s}(a_{\lambda})\rightarrow \phi_{t}(a)p_{s}$ by a similar argument, and therefore each $\phi_{s}$ is extendible.
\end{proof}

\begin{remark}
\label{computation}
Note that, therefore, by Lemma \ref{mux}, we have $\overline{\phi}_{s}=\chi_{s}$ for every $s\in G$. Also, we would like to recall that
\begin{align}
\label{suprem2}
\overline{\phi}_{s}(m)\overline{\phi}_{t}(n)=\overline{\phi}_{s\vee t}(\overline{\alpha}_{(s\vee t)s^{-1}}(m)\overline{\alpha}_{(s\vee t)t^{-1}}(n))
\end{align}
for all $s,t\in G$ and $m,n\in\M(A)$. So, in particular, if $s\leq t$, then, since $s\vee t=t$,
\begin{align}
\label{suprem3}
\overline{\phi}_{s}(m)\overline{\phi}_{t}(n)=\overline{\phi}_{t}(\overline{\alpha}_{ts^{-1}}(m) n),
\end{align}
and similarly,
\begin{align}
\label{suprem4}
\overline{\phi}_{s}(m)\overline{\phi}_{t}(n)=\overline{\phi}_{s}(m \overline{\alpha}_{st^{-1}}(n)),
\end{align}
if $t\leq s$. These equations have key roles in some computations in section \ref{sec:full piso}.
\end{remark}

Next, let $\J$ be the $C^*$-subalgebra of $\B$ generated by $\{\phi_{s}(a)-\phi_{t}(\alpha_{ts^{-1}}(a)): s<t\in G, a\in A\}$.
\begin{prop}
\label{ideal J}
We have
\begin{align}\label{J-span}
\J=\clsp\{\phi_{s}(a)-\phi_{t}(\alpha_{ts^{-1}}(a)): s<t\in G, a\in A\},
\end{align}
which is in fact an essential ideal of $\B$.
\end{prop}

\begin{proof}
Firstly, for all $a,b\in A$ and $r,s,t\in G$ with $s<t$, we have
\begin{eqnarray}
\label{eqt-1}
\begin{array}{l}
\phi_{r}(b)[\phi_{s}(a)-\phi_{t}(\alpha_{ts^{-1}}(a))]\\
=\phi_{r}(b)\phi_{s}(a)-\phi_{r}(b)\phi_{t}(\alpha_{ts^{-1}}(a))\\
=\phi_{r\vee s}\big(\alpha_{(r\vee s)r^{-1}}(b)\alpha_{(r\vee s)s^{-1}}(a)\big)-\phi_{r\vee t}\big(\alpha_{(r\vee t)r^{-1}}(b)\alpha_{(r\vee t)t^{-1}}(\alpha_{ts^{-1}}(a))\big)\\
=\phi_{r\vee s}\big(\alpha_{(r\vee s)r^{-1}}(b)\alpha_{(r\vee s)s^{-1}}(a)\big)-\phi_{r\vee t}\big(\alpha_{(r\vee t)r^{-1}}(b)\alpha_{(r\vee t)s^{-1}}(a)\big)\\
=\phi_{x}(c)-\phi_{y}(\alpha_{yx^{-1}}(c))\in \J,
\end{array}
\end{eqnarray}
where $c=\alpha_{(r\vee s)r^{-1}}(b)\alpha_{(r\vee s)s^{-1}}(a)$, $x=r\vee s$, and $y=r\vee t$ (note that as $s<t$, $x\leq y$).
Now, by applying (\ref{eqt-1}), one can see that, for all $a,b\in A$ and $s,t,x,y\in G$ with $s<t$ and $x<y$, the product
$$[\phi_{s}(a)-\phi_{t}(\alpha_{ts^{-1}}(a))][\phi_{x}(b)-\phi_{y}(\alpha_{yx^{-1}}(b))]$$
of the spanning elements of $\J$ equals the sum of two elements of the same form.
Moreover,
$$[\phi_{s}(a)-\phi_{t}(\alpha_{ts^{-1}}(a))]^{*}
=\phi_{s}(a)^{*}-\phi_{t}(\alpha_{ts^{-1}}(a))^{*}=\phi_{s}(a^{*})-\phi_{t}(\alpha_{ts^{-1}}(a^{*}))\in \J.$$
Therefore, (\ref{J-span}) is indeed true. Note that, (\ref{eqt-1}) also implies that $\J$ is actually an ideal of $\B$.

Finally we show that $\J$ is an essential ideal of $\B$. If $\xi\J=0$ for some $\xi\in\B$, then for each $s\in G$, $$\xi[\phi_{s}(\xi(s)^{*})-\phi_{t}(\alpha_{ts^{-1}}(\xi(s)^{*}))]=0,$$ where $t\in G$ with $s<t$.
So, we must have $\xi(s)\xi(s)^{*}=0$ in $A$ for each $s\in G$, which implies that $\xi(s)=0$. Thus $\xi=0$, and therefore $\J$ is essential.
\end{proof}

Note that a simple calculation shows that
\begin{align}
\label{eq3}
[\phi_{s}(a)-\phi_{t}(\alpha_{ts^{-1}}(a))][\phi_{s}(b)-\phi_{t}(\alpha_{ts^{-1}}(b))]=\phi_{s}(ab)-\phi_{t}(\alpha_{ts^{-1}}(ab))
\end{align}
is valid for all $a,b\in A$ and $s<t\in G$. This equation will be applied later in Lemma \ref{ker}.

For the lemma which follows next, we recall that the dynamical system $(A,P,\alpha)$
gives rise to a directed system $(A_{s},\varphi_{s}^{t})_{s,t\in G}$ such that $A_{s}=A$
for every $s\in G$, and each homomorphism $\varphi_{s}^{t}:A_{s}\rightarrow A_{t}$ is given by $\alpha_{ts^{-1}}$
for all $s,t\in G$ with $s\leq t$. Let $A_{\infty}$ be the direct limit of the directed system.
If $\alpha^{s}:A_{s}\rightarrow A_{\infty}$ is the canonical homomorphism of $A_{s}$ into $A_{\infty}$ for every $s\in G$,
then $\bigcup_{s\in G}\alpha^{s}(A_{s})$ is a dense subalgebra of $A_{\infty}$.
However, since $\alpha^{s}(a)=\alpha^{(s\vee e)}(\alpha_{(s\vee e)s^{-1}}(a))$ for every $s\in G$, it follows
that $A_{\infty}=\overline{\bigcup_{s\in G}\alpha^{s}(A_{s})}=\overline{\bigcup_{s\in P}\alpha^{s}(A_{s})}$.
Moreover, there is an action $\alpha_{\infty}$ of $G$ on the $C^*$-algebra $A_{\infty}$ by automorphisms such that
$$(\alpha_{\infty})_{t}\circ \alpha^{s}=\alpha^{t^{-1}s}$$ for all $s,t\in G$.
Therefore, we obtain an automorphic dynamical system $(A_{\infty},G,\alpha_{\infty})$ for which, we say, it is obtained by the
\emph{dilation} of the (semigroup) dynamical system $(A,P,\alpha)$.
Note that, we might get $0$ $C^*$-algebra for $A_{\infty}$. However, if each endomorphism $\alpha_{x}$ in the
system $(A,P,\alpha)$ is injective, then this ensures that $A_{\infty}\neq 0$. Also, in this case, each canonical homomorphism
$\alpha^{s}:A_{s}\rightarrow A_{\infty}$ clearly becomes an isometry (see more in \cite[Appendix L]{Wegge} or \cite[Proposition 11.4.1]{KR}).
For the following lemma only, we actually require each endomorphism $\alpha_{x}$ to be injective.

\begin{lemma}
\label{B extension}
Let $(G,P)$ be an abelian lattice-ordered group, and $(A,P,\alpha)$ a dynamical system in which the action
$\alpha$ of $P$ is given by extendible injective endomorphisms of $A$. Then, there is a surjective
homomorphism $\sigma:\B\rightarrow A_{\infty}$ such that $\sigma(\phi_{y}(a))=\alpha^{y}(a)$ for every $a\in A$ and $y\in G$. Moreover,
\begin{align}
\label{ker sigma}
\ker \sigma=\{\xi\in\B:\ \textrm{the net}\ \{\|\xi(x)\|\}_{x\in G}\ \textrm{converges to}\ 0\},
\end{align}
which contains the ideal $\J$.
\end{lemma}

\begin{proof}
Define a map of the dense subalgebra $\lsp\{\phi_{y}(a):y\in G, a\in A\}$ of $\B$ into $A_{\infty}$
such that $\sum_{i=0}^{n}\phi_{y_{i}}(a_{i})\mapsto \sum_{i=0}^{n}\alpha^{y_{i}}(a_{i})$. Prior to seeing that this map is well-defined,
we have to show that the representation $\sum_{i=0}^{n}\phi_{y_{i}}(a_{i})$ is a unique representation of that element. To do so,
it is enough to see that if $\sum_{i=0}^{n}\phi_{y_{i}}(a_{i})=0$, then every $a_{i}$ is zero. Thus, suppose that
$\sum_{i=0}^{n}\phi_{y_{i}}(a_{i})=0$. If $y_{j}$ is minimal in $\{y_{0},y_{1},...,y_{n}\}$, then we have
$$a_{j}=\alpha_{y_{j}y_{j}^{-1}}(a_{j})=\sum_{i=0}^{n}\phi_{y_{i}}(a_{i})(y_{j})
=\bigg(\sum_{i=0}^{n}\phi_{y_{i}}(a_{i})\bigg)(y_{j})=0.$$
Moreover, we get $\sum_{i=0}^{n}\phi_{y_{i}}(a_{i})=0$ in which $i\neq j$. Moving inductively we get $a_{i}=0$ for
every $0\leq i\leq n$. Now, the map $\sum_{i=0}^{n}\phi_{y_{i}}(a_{i})\mapsto \sum_{i=0}^{n}\alpha^{y_{i}}(a_{i})$ is well-defined.
This is due to the fact that, if $z=y_{0}\vee y_{1}\vee...\vee y_{n}$, then
\begin{eqnarray*}
\begin{array}{rcl}
\|\sum_{i=0}^{n}\alpha^{y_{i}}(a_{i})\|&=&\|\sum_{i=0}^{n}\alpha^{z}(\alpha_{zy_{i}^{-1}}(a_{i}))\|\\
&=&\|\alpha^{z}\big(\sum_{i=0}^{n}\alpha_{zy_{i}^{-1}}(a_{i})\big)\|\\
&=&\|\sum_{i=0}^{n}\alpha_{zy_{i}^{-1}}(a_{i})\|\\
&=&\|\sum_{i=0}^{n} \phi_{y_{i}}(a_{i})(z)\|\\
&=&\|(\sum_{i=0}^{n}\phi_{y_{i}}(a_{i}))(z)\|\\
&\leq&\|\sum_{i=0}^{n}\phi_{y_{i}}(a_{i})\|.
\end{array}
\end{eqnarray*}
This map is linear and bounded, too, and therefore, it extends to
a bounded linear map $\sigma:\B\rightarrow A_{\infty}$ such that $\sigma(\phi_{y}(a))=\alpha^{y}(a)$. The map
$\sigma$ is obviously surjective. In addition, since
$$\sigma(\phi_{y}(a))^{*}=\alpha^{y}(a)^{*}=\alpha^{y}(a^{*})=\sigma(\phi_{y}(a^{*}))=\sigma(\phi_{y}(a)^{*}),$$ and
\begin{eqnarray*}
\begin{array}{rcl}
\sigma(\phi_{x}(a)\phi_{y}(b))&=&\sigma(\phi_{z}(\alpha_{zx^{-1}}(a)\alpha_{zy^{-1}}(b)))\\
&=&\alpha^{z}(\alpha_{zx^{-1}}(a)\alpha_{zy^{-1}}(b))\\
&=&\alpha^{z}(\alpha_{zx^{-1}}(a))\alpha^{z}(\alpha_{zy^{-1}}(b))\\
&=&\alpha^{x}(a)\alpha^{y}(b)=\sigma(\phi_{x}(a))\sigma(\phi_{y}(b)),
\end{array}
\end{eqnarray*}
where $z=x\vee y$, it follows that $\sigma$ is indeed a surjective $*$-homomorphism.

Next, before we identify $\ker \sigma$, note that, for each spanning element $\phi_{x}(a)-\phi_{y}(\alpha_{yx^{-1}}(a))$ of $\J$,
where $a\in A$ and $x<y\in G$, we have
$$\sigma(\phi_{x}(a)-\phi_{y}(\alpha_{yx^{-1}}(a)))=\alpha^{x}(a)-\alpha^{y}(\alpha_{yx^{-1}}(a))
=\alpha^{y}(yx^{-1}(a))-\alpha^{y}(yx^{-1}(a))=0.$$
Thus, we conclude that the ideal $\J$ is contained in $\ker \sigma$.

Now, to prove (\ref{ker sigma}), first, let $\xi\in \ker \sigma$. Then, for every $\varepsilon>0$, there is a finite sum $\sum_{i=0}^{n}\phi_{y_{i}}(a_{i})$ such that $\|\sum_{i=0}^{n}\phi_{y_{i}}(a_{i})-\xi\|<\varepsilon/2$, and therefore,
\begin{eqnarray}
\label{eq6}
\begin{array}{rcl}
\|\sum_{i=0}^{n}\alpha^{y_{i}}(a_{i})\|&=&\|\sigma(\sum_{i=0}^{n}\phi_{y_{i}}(a_{i}))-\sigma(\xi)\|\\
&=&\|\sigma(\sum_{i=0}^{n}\phi_{y_{i}}(a_{i})-\xi)\|\\
&\leq&\|\sum_{i=0}^{n}\phi_{y_{i}}(a_{i})-\xi\|<\varepsilon/2.
\end{array}
\end{eqnarray}
It thus follows that, if $z=y_{0}\vee y_{1}\vee...\vee y_{n}$, then for every $x\geq z$, we have
\begin{eqnarray}
\label{eq12}
\begin{array}{rcl}
\|\sum_{i=0}^{n}\alpha_{xy_{i}^{-1}}(a_{i})\|&=&\|\sum_{i=0}^{n}\alpha_{xz^{-1}}(\alpha_{zy_{i}^{-1}}(a_{i}))\|\\
&=&\|\alpha_{xz^{-1}}\big(\sum_{i=0}^{n}\alpha_{zy_{i}^{-1}}(a_{i})\big)\|\\
&=&\|\sum_{i=0}^{n}\alpha_{zy_{i}^{-1}}(a_{i})\|\\
&=&\|\alpha^{z}(\sum_{i=0}^{n}\alpha_{zy_{i}^{-1}}(a_{i}))\|\\
&=&\|\sum_{i=0}^{n}\alpha^{z}(\alpha_{zy_{i}^{-1}}(a_{i}))\|\\
&=&\|\sum_{i=0}^{n}\alpha^{y_{i}}(a_{i})\|,
\end{array}
\end{eqnarray}
and hence, by (\ref{eq6}) and (\ref{eq12}), we get
 \begin{align}
\label{eq7}
\|\sum_{i=0}^{n}\alpha_{xy_{i}^{-1}}(a_{i})\|=\|\sum_{i=0}^{n}\alpha^{y_{i}}(a_{i})\|
\leq \|\sum_{i=0}^{n}\phi_{y_{i}}(a_{i})-\xi\|<\varepsilon/2.
\end{align}
Consequently, if $x\geq z$, then by (\ref{eq7}),
\begin{eqnarray*}
\begin{array}{rcl}
\|\xi(x)\|&=&\|\xi(x)-(\sum_{i=0}^{n}\phi_{y_{i}}(a_{i}))(x)+(\sum_{i=0}^{n}\phi_{y_{i}}(a_{i}))(x)\|\\
&=&\|(\xi-\sum_{i=0}^{n}\phi_{y_{i}}(a_{i}))(x)+\sum_{i=0}^{n} \phi_{y_{i}}(a_{i})(x)\|\\
&\leq&\|(\xi-\sum_{i=0}^{n}\phi_{y_{i}}(a_{i}))(x)\|+\|\sum_{i=0}^{n}\alpha_{xy_{i}^{-1}}(a_{i})\|\\
&\leq&\|\xi-\sum_{i=0}^{n}\phi_{y_{i}}(a_{i})\|+\|\sum_{i=0}^{n}\alpha_{xy_{i}^{-1}}(a_{i})\|<\varepsilon/2+\varepsilon/2=\varepsilon.
\end{array}
\end{eqnarray*}
This implies that the net $\{\|\xi(x)\|\}_{x\in G}$ converges to $0$. So, $\ker \sigma$ is contained in the right
hand side of (\ref{ker sigma}). To see the other inclusion,
let $\xi$ be in the right hand side of (\ref{ker sigma}). So, for every $\varepsilon>0$, there exists $s\in G$ such that
$\|\xi(x)\|<\varepsilon/3$ for every $x\geq s$, and a finite sum $\sum_{i=0}^{n}\phi_{y_{i}}(a_{i})$ such that $\|\xi-\sum_{i=0}^{n}\phi_{y_{i}}(a_{i})\|<\varepsilon/3$. Now, if $z=y_{0}\vee y_{1}\vee...\vee y_{n}$, then we have
\begin{eqnarray}
\label{eq8}
\begin{array}{rcl}
\|\sigma(\xi)\|&=&\|\sigma(\xi-\sum_{i=0}^{n}\phi_{y_{i}}(a_{i}))+\sigma(\sum_{i=0}^{n}\phi_{y_{i}}(a_{i}))\|\\
&\leq&\|\sigma(\xi-\sum_{i=0}^{n}\phi_{y_{i}}(a_{i}))\|+\|\sigma(\sum_{i=0}^{n}\phi_{y_{i}}(a_{i}))\|\\
&\leq&\|\xi-\sum_{i=0}^{n}\phi_{y_{i}}(a_{i})\|+\|\sum_{i=0}^{n}\alpha^{y_{i}}(a_{i})\|\\
&<&\varepsilon/3+\|\sum_{i=0}^{n}\alpha^{y_{i}}(a_{i})\|.
\end{array}
\end{eqnarray}
Also, in the bottom line, by the same computation as (\ref{eq12}), we have $\|\sum_{i=0}^{n}\alpha^{y_{i}}(a_{i})\|=\|\sum_{i=0}^{n}\alpha_{xy_{i}^{-1}}(a_{i})\|$ for every $x\geq z$. So it follows that
$$\|\sigma(\xi)\|<\varepsilon/3+\|\sum_{i=0}^{n}\alpha_{xy_{i}^{-1}}(a_{i})\|\ \ \textrm{for all}\ x\geq z.$$
Moreover, for every $x\geq z$,
\begin{eqnarray*}
\begin{array}{rcl}
\|\sum_{i=0}^{n}\alpha_{xy_{i}^{-1}}(a_{i})\|&=&\|\sum_{i=0}^{n}\phi_{y_{i}}(a_{i})(x)-\xi(x)+\xi(x)\|\\
&=&\|(\sum_{i=0}^{n}\phi_{y_{i}}(a_{i}))(x)-\xi(x)+\xi(x)\|\\
&=&\|(\sum_{i=0}^{n}\phi_{y_{i}}(a_{i})-\xi)(x)+\xi(x)\|\\
&\leq&\|(\sum_{i=0}^{n}\phi_{y_{i}}(a_{i})-\xi)(x)\|+\|\xi(x)\|\\
&\leq&\|\sum_{i=0}^{n}\phi_{y_{i}}(a_{i})-\xi\|+\|\xi(x)\|<\varepsilon/3+\varepsilon/3=2\varepsilon/3.\\
\end{array}
\end{eqnarray*}
Therefore, it follows that, if $x\geq z$, then
$$\|\sigma(\xi)\|<\varepsilon/3+\|\sum_{i=0}^{n}\alpha_{xy_{i}^{-1}}(a_{i})\|\leq\varepsilon/3+2\varepsilon/3=\varepsilon.$$
This implies that $\|\sigma(\xi)\|<\varepsilon$ for every $\varepsilon>0$. So, we must have $\|\sigma(\xi)\|=0$, and
therefore, $\sigma(\xi)=0$. This completes the proof.
\end{proof}

\section{The Nica-Toeplitz algebra $\T_{\textrm{cov}}(A\times_{\alpha} P)$ as a full corner in crossed product by group}
\label{sec:full piso}
Consider the semigroup dynamical system $(A,P,\alpha)$ in which the action $\alpha$ of $P$ is given by extendible endomorphisms of the $C^*$-algebra $A$. Let $\B$ be the $C^*$-algebra constructed corresponding to the system $(A,P,\alpha)$ in section \ref{sec:alg B}.
There is an action $\beta$ of $G$ by automorphisms of $\B$ induced by the shift on $\ell^{\infty}(G,A)$,
such that $\beta_{t}\circ\phi_{s}=\phi_{ts}$ for all $s,t\in G$.
Thus we obtain a group dynamical system $(\B, G,\beta)$. Define a map $\rho:\B\rightarrow \L(\ell^{2}(G,A))$ by $(\rho(\xi)f)(x)=\xi(x)f(x)$, and $U:G\rightarrow \L(\ell^{2}(G,A))$ by $(U_{t}f)(x)=f(t^{-1}x)$, where $\xi\in\B$ and $f\in\ell^{2}(G,A)$. Then $\rho$
is a nondegenerate representation, and $U$ is a unitary representation such that we have $\rho(\beta_{t}(\xi))=U_{t}\rho(\xi)U_{t}^{*}$. Therefore the pair $(\rho,U)$ is a covariant representation of $(\B,G,\beta)$ on $\ell^{2}(G,A)$.
Moreover, let $(\B\rtimes_{\beta}G,j_{\B},j_{G})$ be the group crossed product associated to the system $(\B,G,\beta)$.
Since for each spanning element $\phi_{s}(a)-\phi_{t}(\alpha_{ts^{-1}}(a))$ of $\J$ and $r\in G$ we have
\begin{eqnarray*}
\begin{array}{rcl}
\beta_{r}\big(\phi_{s}(a)-\phi_{t}(\alpha_{ts^{-1}}(a))\big)&=&\beta_{r}\big(\phi_{s}(a)\big)-\beta_{r}\big(\phi_{t}(\alpha_{ts^{-1}}(a))\big)\\
&=&\phi_{rs}(a)-\phi_{rt}(\alpha_{ts^{-1}}(a))\\
&=&\phi_{rs}(a)-\phi_{rt}(\alpha_{(rt)(rs)^{-1}}(a))\in \J,
\end{array}
\end{eqnarray*}
$\J$ is actually a $\beta$-invariant essential ideal of $\B$. It thus follows that $\J\rtimes_{\beta} G$ sits
in $\B\rtimes_{\beta} G$ as an essential ideal \cite[Proposition 2.4]{Kusuda}. Now, we have:

\begin{theorem}
\label{main}
Suppose that $(A,P,\alpha)$ is a dynamical system consisting of a $C^{*}$-algebra $A$ and an action $\alpha$ of $P$ by extendible endomorphisms of $A$. Let $p=\overline{j_{\B}\circ\phi_{e}}(1)$, and
$$k_{A}:A\rightarrow p(\B\rtimes_{\beta} G)p\ \ \textrm{and}\ \ k_{P}:P\rightarrow \mathcal{M}(p(\B\rtimes_{\beta} G)p)$$
be the maps defined by $k_{A}(a)=(j_{\B}\circ\phi_{e})(a)$ and $k_{P}(x)=pj_{G}(x)^{*}p$ for all $a\in A$ and $x\in P$. Then
the triple $(p(\B\rtimes_{\beta} G)p,k_{A},k_{P})$ is a Nica-Toeplitz crossed product for $(A,P,\alpha)$, and hence
$(\T_{\textrm{cov}}(A\times_{\alpha} P),i_{A},i_{P})\simeq (p(\B\rtimes_{\beta} G)p,k_{A},k_{P})$.
Moreover, $\T_{\textrm{cov}}(A\times_{\alpha} P)$ is a full corner in $\B\rtimes_{\beta} G$.
\end{theorem}

\begin{proof}
First of all, since $j_{\B}$ and $\phi_{e}$ are injective, so is $k_{A}$. Let $\Lambda$ be a nondegenerate representation of
$p(\B\rtimes_{\beta} G)p$ on a Hilbert space $H$. We show that $(\pi,W):=(\Lambda\circ k_{A}, \overline{\Lambda}\circ k_{P})$ is a
Nica-Toeplitz covariant representation of $(A,P,\alpha)$ on $H$. Take an approximate identity $\{a_{\lambda}\}$ in $A$. Since $\phi_{e}(a_{\lambda})\rightarrow\overline{\phi}_{e}(1)$ strictly in $\M(\B)$, and $j_{\B}$ is nondegenerate, we get
$k_{A}(a_{\lambda})\rightarrow\overline{j_{\B}}(\overline{\phi}_{e}(1))$ strictly in $\mathcal{M}(\B\rtimes_{\beta} G)$,
where $\overline{j_{\B}}(\overline{\phi}_{e}(1))=\overline{j_{\B}\circ\phi_{e}}(1)=p$. Therefore, as $\Lambda$ is nondegenerate,
$\pi(a_{\lambda})=\Lambda(k_{A}(a_{\lambda}))$ converges strictly to $\overline{\Lambda}(p)=1$
strictly in $B(H)=\M(\K(H))$, which implies that $\pi$ is nondegenerate. Next, we show that $W:P\rightarrow B(H)$ is a
partial-isometric representation of $P$ on $H$ which satisfies the equation
$$W_{x}^{*}W_{x}W_{y}^{*}W_{y}=W_{x\vee y}^{*}W_{x\vee y}$$ for all $x,y\in P$. To see that each $W_{x}$ is a partial-isometry,
note that, by applying the covariance equation of the pair $(j_{\B},j_{G})$, we have
\begin{eqnarray}\label{eq10}
\begin{array}{rcl}
k_{P}(x)k_{P}(x)^{*}k_{P}(x)&=&pj_{G}(x)^{*}pj_{G}(x)\overline{j_{\B}}(\overline{\phi}_{e}(1))j_{G}(x)^{*}p\\
&=&pj_{G}(x)^{*}p\overline{j_{\B}}(\overline{\beta}_{x}(\overline{\phi}_{e}(1)))p\\
&=&pj_{G}(x)^{*}p\overline{j_{\B}}(\overline{\phi}_{x}(1))p\\
&=&pj_{G}(x)^{*}\overline{j_{\B}}(\overline{\phi}_{e}(1))\overline{j_{\B}}(\overline{\phi}_{x}(1))p\\
&=&pj_{G}(x)^{*}\overline{j_{\B}}(\overline{\phi}_{e}(1)\overline{\phi}_{x}(1))p\\
\end{array}
\end{eqnarray}
Then, in the bottom line, for $\overline{\phi}_{e}(1)\overline{\phi}_{x}(1)$, since $e\leq x$, by (\ref{suprem3}), we have
$$\overline{\phi}_{e}(1)\overline{\phi}_{x}(1)=\overline{\phi}_{x}(\overline{\alpha}_{x}(1)),$$ and therefore
$$k_{P}(x)k_{P}(x)^{*}k_{P}(x)=pj_{G}(x)^{*}\overline{j_{\B}}(\overline{\phi}_{x}(\overline{\alpha}_{x}(1)))p$$
Now, again, by the covariance equation of $(j_{\B},j_{G})$,
\begin{eqnarray}\label{eq11}
\begin{array}{rcl}
k_{P}(x)k_{P}(x)^{*}k_{P}(x)&=&p\overline{j_{\B}}(\overline{\phi}_{e}(\overline{\alpha}_{x}(1)))j_{G}(x)^{*}p\\
&=&q\overline{j_{\B}}(\overline{\phi}_{e}(1)\overline{\phi}_{x^{-1}}(1))j_{G}(x)^{*}p\ \ \ [\textrm{by (\ref{suprem4}),}\ \textrm{as}\ x^{-1}\leq e]\\
&=&p\overline{j_{\B}}(\overline{\phi}_{e}(1))\overline{j_{\B}}(\overline{\phi}_{x^{-1}}(1))j_{G}(x)^{*}p\\
&=&p\overline{j_{\B}}(\overline{\phi}_{x^{-1}}(1))j_{G}(x)^{*}p\\
&=&pj_{G}(x)^{*}\overline{j_{\B}}(\overline{\beta}_{x}(\overline{\phi}_{x^{-1}}(1)))p\\
&=&pj_{G}(x)^{*}\overline{j_{\B}}(\overline{\phi}_{e}(1))p=pj_{G}(x)^{*}p=k_{P}(x).\\
\end{array}
\end{eqnarray}
Therefore, it follows that
$$W_{x}W_{x}^{*}W_{x}=\overline{\Lambda}(k_{P}(x)k_{P}(x)^{*}k_{P}(x))=\overline{\Lambda}(k_{P}(x))=W_{x},$$
which means that each $W_{x}$ is a partial-isometry. Moreover,
\begin{eqnarray*}
\begin{array}{rcl}
k_{P}(x)k_{P}(y)&=&pj_{G}(x)^{*}\overline{j_{\B}}(\overline{\phi}_{e}(1))j_{G}(y)^{*}p\\
&=&pj_{G}(x)^{*}j_{G}(y)^{*}\overline{j_{\B}}(\overline{\beta}_{y}(\overline{\phi}_{e}(1)))p\\
&=&pj_{G}(x y)^{*}\overline{j_{\B}}(\overline{\phi}_{y}(1))p\\
&=&p\overline{j_{\B}}(\overline{\phi}_{e}(1))j_{G}(x y)^{*}\overline{j_{\B}}(\overline{\phi}_{y}(1))p\\
&=&pj_{G}(x y)^{*}\overline{j_{\B}}(\overline{\beta}_{x y}(\overline{\phi}_{e}(1)))\overline{j_{\B}}(\overline{\phi}_{y}(1))p\\
&=&pj_{G}(x y)^{*}\overline{j_{\B}}(\overline{\phi}_{x y}(1))\overline{j_{\B}}(\overline{\phi}_{y}(1))p\\
&=&pj_{G}(x y)^{*}\overline{j_{\B}}(\overline{\phi}_{x y}(1)\overline{\phi}_{y}(1))p\\
&=&pj_{G}(x y)^{*}\overline{j_{\B}}(\overline{\phi}_{x y}(\overline{\alpha}_{x}(1)))\overline{j_{\B}}(\overline{\phi}_{e}(1))p
\ \ \ [\textrm{by (\ref{suprem4}),}\ \textrm{as}\ y\leq xy]\\
&=&pj_{G}(x y)^{*}\overline{j_{\B}}(\overline{\phi}_{x y}(\overline{\alpha}_{x}(1))\overline{\phi}_{e}(1))p\\
&=&pj_{G}(x y)^{*}\overline{j_{\B}}(\overline{\phi}_{x y}(\overline{\alpha}_{x}(1)\overline{\alpha}_{x y}(1)))p
\ \ \ [\textrm{by (\ref{suprem4}),}\ \textrm{as}\ e\leq xy]\\
&=&pj_{G}(x y)^{*}\overline{j_{\B}}(\overline{\phi}_{x y}(\overline{\alpha}_{x y}(1)))p\\
&=&p\overline{j_{\B}}(\overline{\phi}_{e}(\overline{\alpha}_{x y}(1)))j_{G}(x y)^{*}p
\end{array}
\end{eqnarray*}
Now, for the bottom line, if we continue the computation similar to (\ref{eq11}) (see that $p\overline{j_{\B}}(\overline{\phi}_{e}(\overline{\alpha}_{x}(1)))j_{G}(x)^{*}p=k_{P}(x)$), then we get
$$k_{P}(x)k_{P}(y)=p\overline{j_{\B}}(\overline{\phi}_{e}(\overline{\alpha}_{x y}(1)))j_{G}(x y)^{*}p=k_{P}(xy).$$
Thus, $$W_{x}W_{y}=\overline{\Lambda}(k_{P}(x)k_{P}(y))=\overline{\Lambda}(k_{P}(xy))=W_{xy}.$$
To see that $W_{x}^{*}W_{x}W_{y}^{*}W_{y}=W_{x\vee y}^{*}W_{x\vee y}$, note that, by a similar computation done in (\ref{eq10}), we have
\begin{align}
\label{eq13}
k_{P}(x)^{*}k_{P}(x)=p\overline{j_{\B}}(\overline{\phi}_{x}(1))p=p\overline{j_{\B}}(\overline{\phi}_{x}(\overline{\alpha}_{x}(1)))p.
\end{align}
So, it follows that
\begin{eqnarray*}
\begin{array}{rcl}
k_{P}(x)^{*}k_{P}(x)k_{P}(y)^{*}k_{P}(y)&=&p\overline{j_{\B}}(\overline{\phi}_{x}(1))p\overline{j_{\B}}(\overline{\phi}_{y}(1))p\\
&=&p\overline{j_{\B}}(\overline{\phi}_{x}(1))\overline{j_{\B}}(\overline{\phi}_{e}(1))\overline{j_{\B}}(\overline{\phi}_{y}(1))p\\
&=&p\overline{j_{\B}}(\overline{\phi}_{x}(1)\overline{\phi}_{e}(1)\overline{\phi}_{y}(1))p\\
&=&p\overline{j_{\B}}(\overline{\phi}_{x}(1)\overline{\phi}_{y}(\overline{\alpha}_{y}(1))p
\ \ \ [\textrm{by (\ref{suprem3}),}\ \textrm{as}\ e\leq y]\\
&=&p\overline{j_{\B}}(\overline{\phi}_{x\vee y}(\overline{\alpha}_{(x\vee y)x^{-1}}(1)
\overline{\alpha}_{(x\vee y)}(1)))p
\ \ \ [\textrm{by (\ref{suprem2})}]\\
&=&p\overline{j_{\B}}(\overline{\phi}_{x\vee y}(\overline{\alpha}_{(x\vee y)}(1)))p
\ \ \ [\textrm{since}\ (x\vee y)x^{-1}\leq (x\vee y)]\\
&=&k_{P}(x\vee y)^{*}k_{P}(x\vee y).\ \ \ [\textrm{by (\ref{eq13})}]
\end{array}
\end{eqnarray*}
Therefore, we have
\begin{eqnarray*}
\begin{array}{rcl}
W_{x}^{*}W_{x}W_{y}^{*}W_{y}&=&\overline{\Lambda}(k_{P}(x)^{*}k_{P}(x)k_{P}(y)^{*}k_{P}(y))\\
&=&\overline{\Lambda}(k_{P}(x\vee y)^{*}k_{P}(x\vee y))\\
&=&\overline{\Lambda}(k_{P}(x\vee y))^{*}\overline{\Lambda}(k_{P}(x\vee y))=W_{x\vee y}^{*}W_{x\vee y}.
\end{array}
\end{eqnarray*}

Now, we show that the pair $(\pi,W)$ satisfies the covariance equations
$$\pi(\alpha_{x}(a))=W_{x}\pi_{A}(a)W_{x}^{*}\ \ \textrm{and}\ \ W_{x}^{*}W_{x}\pi(a)=\pi(a)W_{x}^{*}W_{x}$$ for every $a\in A$ and $x\in P$. We have

\begin{eqnarray*}
\begin{array}{rcl}
k_{P}(x)k_{A}(a)k_{P}(x)^{*}&=&pj_{G}(x)^{*}j_{\B}(\phi_{e}(a))j_{G}(x)p\\
&=&pj_{\B}(\beta_{x^{-1}}(\phi_{e}(a)))p\ \ [\textrm{by the covariance of}\ (j_{\B},j_{G})]\\
&=&pj_{\B}(\phi_{x^{-1}}(a))p\\
&=&p\overline{j_{\B}}(\overline{\phi}_{e}(1))j_{\B}(\phi_{x^{-1}}(a))p\\
&=&pj_{\B}(\overline{\phi}_{e}(1)\phi_{x^{-1}}(a))p\\
&=&pj_{\B}(\phi_{e}(\alpha_{x}(a)))p\ \ \ [\textrm{by (\ref{suprem4}),}\ \textrm{as}\ x^{-1}\leq e]\\
&=&(j_{\B}\circ\phi_{e})(\alpha_{x}(a))=k_{A}(\alpha_{x}(a)).\\
\end{array}
\end{eqnarray*}
Thus, it follows that $$\pi(\alpha_{x}(a))=\Lambda(k_{A}(\alpha_{x}(a)))=\Lambda(k_{P}(x)k_{A}(a)k_{P}(x)^{*})=W_{x}\pi(a)W_{x}^{*}.$$
To see that $W_{x}^{*}W_{x}\pi(a)=\pi(a)W_{x}^{*}W_{x}$, first, by using (\ref{eq13}), we have
\begin{eqnarray*}
\begin{array}{rcl}
k_{P}(x)^{*}k_{P}(x)k_{A}(a)&=&p\overline{j_{\B}}(\overline{\phi}_{x}(1))pj_{\B}(\phi_{e}(a))\\
&=&p\overline{j_{\B}}(\overline{\phi}_{x}(1))j_{\B}(\phi_{e}(a))p\\
&=&pj_{\B}(\overline{\phi}_{x}(1)\phi_{e}(a))p\\
&=&pj_{\B}(\phi_{x}(1\alpha_{x}(a)))p\ \ \ [\textrm{by (\ref{suprem4}),}\ \textrm{as}\ e\leq x]\\
&=&pj_{\B}(\phi_{x}(\alpha_{x}(a)1))p\\
&=&pj_{\B}(\phi_{e}(a)\overline{\phi}_{x}(1))p\\
&=&pj_{\B}(\phi_{e}(a))\overline{j_{\B}}(\overline{\phi}_{x}(1))p\\
&=&j_{\B}(\phi_{e}(a))p\overline{j_{\B}}(\overline{\phi}_{x}(1))p=k_{A}(a)k_{P}(x)^{*}k_{P}(x).
\end{array}
\end{eqnarray*}
Therefore, we get
\begin{eqnarray*}
\begin{array}{rcl}
W_{x}^{*}W_{x}\pi(a)&=&\Lambda(k_{P}(x)^{*}k_{P}(x)k_{A}(a))\\
&=&\Lambda(k_{A}(a)k_{P}(x)^{*}k_{P}(x))=\pi(a)W_{x}^{*}W_{x}.
\end{array}
\end{eqnarray*}
So, the pair $(\pi,W)$ is a Nica-Toeplitz covariant representation of $(A,P,\alpha)$ on $H$.

Next, we want to prove that
\begin{align}
\label{span}
p(\B\rtimes_{\beta} G)p=\clsp\{k_{P}(x)^{*}k_{A}(a)k_{P}(y): a\in A, x,y\in P\}.
\end{align}
We only need to show that $p(\B\rtimes_{\beta} G)p$ is a subset of the right hand side of (\ref{span}), as the other inclusion is obvious. So, recall that, since
elements of the form $j_{\B}(\phi_{r}(a))j_{G}(s)$ span $\B\rtimes_{\beta} G$, where $a\in A$ and $r,s\in G$, $p(\B\rtimes_{\beta} G)p$ is spanned by the elements $pj_{\B}(\phi_{r}(a))j_{G}(s)p$. However,
\begin{eqnarray*}
\begin{array}{rcl}
pj_{\B}(\phi_{r}(a))j_{G}(s)p&=&p\overline{j_{\B}}(\overline{\phi}_{e}(1))j_{\B}(\phi_{r}(a))j_{G}(s)p\\
&=&pj_{\B}(\overline{\phi}_{e}(1)\phi_{r}(a))j_{G}(s)p\\
&=&pj_{\B}(\phi_{(e\vee r)}(\overline{\alpha}_{(e\vee r)}(1)\alpha_{(e\vee r)r^{-1}}(a)))j_{G}(s)p,\ \ \ [\textrm{by (\ref{suprem2})}]
\end{array}
\end{eqnarray*}
where $(e\vee r)\in P$, as $e\leq e\vee r$. So, for the spanning elements $pj_{\B}(\phi_{r}(a))j_{G}(s)p$ of $p(\B\rtimes_{\beta} G)p$, we can assume that $r\in P$ without loss of generality. Furthermore,

\begin{eqnarray*}
\begin{array}{rcl}
pj_{\B}(\phi_{r}(a))j_{G}(s)p&=&pj_{G}(s)j_{\B}(\beta_{s^{-1}}(\phi_{r}(a)))p\\
&=&pj_{G}(s)j_{\B}(\phi_{s^{-1}r}(a))\overline{j_{\B}}(\overline{\phi}_{e}(1))p\\
&=&pj_{G}(s)j_{\B}(\phi_{s^{-1}r}(a)\overline{\phi}_{e}(1))p\\
&=&pj_{G}(s)j_{\B}(\phi_{(s^{-1}r)\vee e}(b))p\ \ \ \ [b:=\alpha_{(s^{-1}r\vee e)r^{-1}s}(a)\overline{\alpha}_{(s^{-1}r\vee e)}(1)]\\
&=&pj_{G}(s)j_{\B}((\beta_{(s^{-1}r)\vee e}\circ\phi_{e})(b))p\\
&=&pj_{G}(s)j_{G}((s^{-1}r)\vee e)(j_{\B}\circ\phi_{e})(b)j_{G}((s^{-1}r)\vee e)^{*}p\\
&=&pj_{G}(s(s^{-1}r\vee e))j_{\B}(\phi_{e}(b))j_{G}((s^{-1}r)\vee e)^{*}p\\
&=&pj_{G}(x)j_{\B}(\phi_{e}(b))j_{G}(y)^{*}p,\\
\end{array}
\end{eqnarray*}
where $x=s(s^{-1}r\vee e)$ and $y=(s^{-1}r)\vee e$. Then, $y\in P$, obviously, and since $s^{-1}r\leq s^{-1}r\vee e$, we have
$$e\leq r=s(s^{-1}r)\leq s(s^{-1}r\vee e)=x.$$ It follows that $e\leq x$, and hence $x\in P$, too. Therefore, we have
$$pj_{\B}(\phi_{r}(a))j_{G}(s)p=pj_{G}(x)j_{\B}(\phi_{e}(b))j_{G}(y)^{*}p=k_{P}(x)^{*}k_{A}(b)k_{P}(y)$$ for some $b\in A$ and $x,y\in P$. This implies that
$p(\B\rtimes_{\beta} G)p$ is a subset of the right hand side of (\ref{span}), and therefore, (\ref{span}) is indeed valid.

Now, suppose that $(\sigma,V)$ is a Nica-Toeplitz covariant representation of $(A,P,\alpha)$ on a Hilbert space $K$. We show that there is a nondegenerate representation $\Phi$ of $p(\B\rtimes_{\beta} G)p$ on $K$ such that
$\Phi\circ k_{A}=\sigma$ and $\overline{\Phi}\circ k_{P}=V$. To do so, first, we take a faithful and nondegenerate representation $\Delta$ of $p(\B\rtimes_{\beta} G)p$ on a Hilbert space $H$. Then, similar to the earlier argument,
one can see that the pair $(\pi,W):=(\Delta\circ k_{A},\overline{\Delta}\circ k_{P})$ is a Nica-Toeplitz covariant representation of $(A,P,\alpha)$ on $H$. Let $\pi\times W$ be the associated
nondegenerate representation (integrated form) of the Nica-Toeplitz algebra $(\T_{\textrm{cov}}(A\times_{\alpha} P),i_{A},i_{P})$ on $H$, such that $(\pi\times W)\circ i_{A}=\pi$ and
$\overline{(\pi\times W)}\circ i_{P}=W$. We claim that $\pi\times W$ is faithful. To prove our claim, we apply \cite[Theorem 9.3]{Fowler}. So, take any finite subset $F=\{x_{1}, x_{2}, ..., x_{n}\}$ of $P\backslash \{e\}$, and assume that
$$\pi(a)\prod_{i=1}^{n}(1-W_{x_{i}}^{*}W_{x_{i}})=0,$$ where $a\in A$. It follows that
$$0=\pi(a)\prod_{i=1}^{n}(1-W_{x_{i}}^{*}W_{x_{i}})=\Delta(k_{A}(a)\prod_{i=1}^{n}(p-k_{P}(x_{i})^{*}k_{P}(x_{i}))),$$ and since $\Delta$ is faithful, we must have
$$k_{A}(a)\prod_{i=1}^{n}(p-k_{P}(x_{i})^{*}k_{P}(x_{i}))=0.$$ However,

\begin{eqnarray*}
\begin{array}{l}
k_{A}(a)\prod_{i=1}^{n}(p-k_{P}(x_{i})^{*}k_{P}(x_{i}))\\
=j_{\B}(\phi_{e}(a))\prod_{i=1}^{n}[p-p\overline{j_{\B}}(\overline{\phi}_{x_{i}}(\overline{\alpha}_{x_{i}}(1)))p]\ \ \ [\textrm{by (\ref{eq13})}]\\
=j_{\B}(\phi_{e}(a))\prod_{i=1}^{n}[p-p\overline{j_{\B}}(\overline{\phi}_{x_{i}}(\overline{\alpha}_{x_{i}}(1)))]\\
=j_{\B}(\phi_{e}(a))[p-p\overline{j_{\B}}(\overline{\phi}_{x_{1}}(\overline{\alpha}_{x_{1}}(1)))]\prod_{i=2}^{n}[p-p\overline{j_{\B}}(\overline{\phi}_{x_{i}}(\overline{\alpha}_{x_{i}}(1)))]\\
=[j_{\B}(\phi_{e}(a))-j_{\B}(\phi_{e}(a))\overline{j_{\B}}(\overline{\phi}_{x_{1}}(\overline{\alpha}_{x_{1}}(1)))]\prod_{i=2}^{n}[p-p\overline{j_{\B}}(\overline{\phi}_{x_{i}}(\overline{\alpha}_{x_{i}}(1)))]\\
=[j_{\B}(\phi_{e}(a))-j_{\B}(\phi_{e}(a)\overline{\phi}_{x_{1}}(\overline{\alpha}_{x_{1}}(1)))]\prod_{i=2}^{n}[p-p\overline{j_{\B}}(\overline{\phi}_{x_{i}}(\overline{\alpha}_{x_{i}}(1)))]\\
=[j_{\B}(\phi_{e}(a))-j_{\B}(\phi_{x_{1}}(\alpha_{x_{1}}(a)))]\prod_{i=2}^{n}[p-p\overline{j_{\B}}(\overline{\phi}_{x_{i}}(\overline{\alpha}_{x_{i}}(1)))]\ \ \ [\textrm{by (\ref{suprem3}),}\ \textrm{as}\ e< x_{1}]\\
=j_{\B}(\phi_{e}(a)-\phi_{x_{1}}(\alpha_{x_{1}}(a)))\prod_{i=2}^{n}[p-p\overline{j_{\B}}(\overline{\phi}_{x_{i}}(\overline{\alpha}_{x_{i}}(1)))]\\
=j_{\B}(\phi_{e}(a)-\phi_{x_{1}}(\alpha_{x_{1}}(a)))\prod_{i=2}^{n}[p-\overline{j_{\B}}(\overline{\phi}_{x_{i}}(\overline{\alpha}_{x_{i}}(1)))]\\
=j_{\B}(\phi_{e}(a)-\phi_{x_{1}}(\alpha_{x_{1}}(a)))\prod_{i=2}^{n}[\overline{j_{\B}}(\overline{\phi}_{e}(1))-\overline{j_{\B}}(\overline{\phi}_{x_{i}}(\overline{\alpha}_{x_{i}}(1)))]\\
=j_{\B}(\phi_{e}(a)-\phi_{x_{1}}(\alpha_{x_{1}}(a)))\prod_{i=2}^{n}\overline{j_{\B}}(\overline{\phi}_{e}(1)-\overline{\phi}_{x_{i}}(\overline{\alpha}_{x_{i}}(1)))\\
=j_{\B}[(\phi_{e}(a)-\phi_{x_{1}}(\alpha_{x_{1}}(a)))\prod_{i=2}^{n}(\overline{\phi}_{e}(1)-\overline{\phi}_{x_{i}}(\overline{\alpha}_{x_{i}}(1)))].
\end{array}
\end{eqnarray*}
Hence,
\begin{align}
\label{inject}
j_{\B}[(\phi_{e}(a)-\phi_{x_{1}}(\alpha_{x_{1}}(a)))\prod_{i=2}^{n}(\overline{\phi}_{e}(1)-\overline{\phi}_{x_{i}}(\overline{\alpha}_{x_{i}}(1)))]=0.
\end{align}
Also, note that, by some calculation, it is not difficult to see that, in fact,
$$(\phi_{e}(a)-\phi_{x_{1}}(\alpha_{x_{1}}(a)))\prod_{i=2}^{n}(\overline{\phi}_{e}(1)-\overline{\phi}_{x_{i}}(\overline{\alpha}_{x_{i}}(1)))=\phi_{e}(a)+\sum_{k=1}^{m}\phi_{y_{k}}(a_{k}),$$
where $a_{k}\in A$ and $y_{k}\in P\backslash \{e\}$ for every $1\leq k\leq m$. Thus, it follows from (\ref{inject}) that
\begin{align}
\label{inject2}
j_{\B}(\phi_{e}(a)+\sum_{k=1}^{m}\phi_{y_{k}}(a_{k}))=0.
\end{align}
Now, if $\rho\times U$ is the nondegenerate representation of $\B\rtimes_{\beta}G$ corresponding to the covariant pair $(\rho,U)$ of
$(\B,G,\beta)$ in $\L(\ell^{2}(G,A))$ (see the beginning of the current section), then, by (\ref{inject2}), we get
$$\rho\times U[j_{\B}(\phi_{e}(a)+\sum_{k=1}^{m}\phi_{y_{k}}(a_{k}))]=\rho(\phi_{e}(a)+\sum_{k=1}^{m}\phi_{y_{k}}(a_{k}))=0$$
in $\L(\ell^{2}(G,A))$. It follows that, for $\varepsilon_{e}\otimes a^{*}\in\ell^{2}(G)\otimes A\simeq\ell^{2}(G,A)$, we must have $$\rho(\phi_{e}(a)+\sum_{k=1}^{m}\phi_{y_{k}}(a_{k}))(\varepsilon_{e}\otimes a^{*})=0.$$
But, since each $y_{k}$ satisfies $e<y_{k}$,
\begin{eqnarray*}
\begin{array}{rcl}
0&=&\rho(\phi_{e}(a)+\sum_{k=1}^{m}\phi_{y_{k}}(a_{k}))(\varepsilon_{e}\otimes a^{*})\\
&=&(\rho(\phi_{e}(a))+\sum_{k=1}^{m}\rho(\phi_{y_{k}}(a_{k})))(\varepsilon_{e}\otimes a^{*})\\
&=&\rho(\phi_{e}(a))(\varepsilon_{e}\otimes a^{*})+\sum_{k=1}^{m}\rho(\phi_{y_{k}}(a_{k}))(\varepsilon_{e}\otimes a^{*})\\
&=&\varepsilon_{e}\otimes aa^{*}+0=\varepsilon_{e}\otimes aa^{*}.
\end{array}
\end{eqnarray*}
Therefore, $\varepsilon_{e}\otimes aa^{*}=0$, and hence, $aa^{*}=0$, which implies that $a=0$. Thus, by \cite[Theorem 9.3]{Fowler}, $\pi\times W$ is faithful. Then, define a map
$$\Psi_{0}:\lsp\{i_{P}(x)^{*}i_{A}(a)i_{P}(y): a\in A, x,y\in P\}\rightarrow p(\B\rtimes_{\beta} G)p$$ by
$$\Psi_{0}(\sum i_{P}(x_{m})^{*}i_{A}(a_{m,n})i_{P}(y_{n}))=\sum k_{P}(x_{m})^{*}k_{A}(a_{m,n})k_{P}(y_{n}).$$ One can see that $\Psi_{0}$ is linear. Also, the following computation
\begin{eqnarray*}
\begin{array}{rcl}
\|\sum k_{P}(x_{m})^{*}k_{A}(a_{m,n})k_{P}(y_{n})\|&=&\|\Delta(\sum k_{P}(x_{m})^{*}k_{A}(a_{m,n})k_{P}(y_{n}))\|\\
&=&\|\sum W_{x_{m}}^{*}\pi(a_{m,n})W_{y_{n}}\|\\
&=&\|\pi\times W(\sum i_{P}(x_{m})^{*}i_{A}(a_{m,n})i_{P}(y_{n}))\|\\
&=&\|\sum i_{P}(x_{m})^{*}i_{A}(a_{m,n})i_{P}(y_{n})\|\ \ [\textrm{as}\ \pi\times W\ \textrm{is faithful}]
\end{array}
\end{eqnarray*}
shows that $\Psi_{0}$ preserves the norm. It therefore follows that it is a well-defined linear map on the dense subspace $\lsp\{i_{P}(x)^{*}i_{A}(a)i_{P}(y): a\in A, x,y\in P\}$ of $\T_{\textrm{cov}}(A\times_{\alpha} P)$. Hence, it
extends to a norm-preserving linear map $\Psi$ of $\T_{\textrm{cov}}(A\times_{\alpha} P)$ into $p(\B\rtimes_{\beta} G)p$. Then, it is not difficult to see that $\Psi$ preserves the involution, too. Moreover, one can calculate to verify that
we have
\begin{eqnarray*}
\begin{array}{l}
\Psi([i_{P}(x)^{*}i_{A}(a)i_{P}(y)][i_{P}(s)^{*}i_{A}(b)i_{P}(t)])\\
=\Psi(i_{P}(x(y\vee s)y^{-1})^{*}i_{A}(\alpha_{(y\vee s)y^{-1}}(a)\overline{\alpha}_{(y\vee s)}(1)\alpha_{(y\vee s)s^{-1}}(b))i_{P}((y\vee s)s^{-1}t))\\
=k_{P}(x(y\vee s)y^{-1})^{*}k_{A}(\alpha_{(y\vee s)y^{-1}}(a)\overline{\alpha}_{(y\vee s)}(1)\alpha_{(y\vee s)s^{-1}}(b))k_{P}((y\vee s)s^{-1}t)\\
=[k_{P}(x)^{*}k_{A}(a)k_{P}(y)][k_{P}(s)^{*}k_{A}(b)k_{P}(t)]=\Psi(i_{P}(x)^{*}i_{A}(a)i_{P}(y))\Psi(i_{P}(s)^{*}i_{A}(b)i_{P}(t))
\end{array}
\end{eqnarray*}
on the spanning elements of $\T_{\textrm{cov}}(A\times_{\alpha} P)$. So, this implies that $\Psi$ also preserves the multiplication, and thus,
it is indeed an injective $*$-homomorphism of $\T_{\textrm{cov}}(A\times_{\alpha} P)$ into $p(\B\rtimes_{\beta} G)p$. Moreover, since
$$k_{P}(x)^{*}k_{A}(a)k_{P}(y)=\Psi(i_{P}(x)^{*}i_{A}(a)i_{P}(y))\in \Psi(\T_{\textrm{cov}}(A\times_{\alpha} P)),$$ it follows by (\ref{span}) that $\Psi$ is also onto. Therefore, $\Psi$ is actually an isomorphism of
$\T_{\textrm{cov}}(A\times_{\alpha} P)$ onto $p(\B\rtimes_{\beta} G)p$. Finally, if $\sigma\times V$ is the associated nondegenerate representation (integrated form) of $\T_{\textrm{cov}}(A\times_{\alpha} P)$ on $K$ such that
$(\sigma\times V)\circ i_{A}=\sigma$ and $\overline{(\sigma\times V)}\circ i_{P}=V$, then $\Phi:=(\sigma\times V)\circ \Psi^{-1}$ is the desired nondegenerate representation of
$p(\B\rtimes_{\beta} G)p$ on $K$ which satisfies $\Phi\circ k_{A}=\sigma$ and $\overline{\Phi}\circ k_{P}=V$.

In order to see that $\T_{\textrm{cov}}(A\times_{\alpha} P)$ is a full corner in $\B\rtimes_{\beta} G$, we have to show that $(\B\rtimes_{\beta} G)p(\B\rtimes_{\beta} G)$ is dense in $\B\rtimes_{\beta} G$.
If $\{a_{\lambda}\}$ is an approximate identity in $A$, then for any spanning element $j_{\B}(\phi_{r}(a))j_{G}(s)$ of $\B\rtimes_{\beta} G$, where $a\in A$ and $r,s\in G$,
we have $j_{\B}(\phi_{r}(a_{\lambda}a))j_{G}(s)\rightarrow j_{\B}(\phi_{r}(a))j_{G}(s)$ in the norm topology of $\B\rtimes_{\beta} G$. Moreover,
\begin{eqnarray*}
\begin{array}{rcl}
j_{\B}(\phi_{r}(a_{\lambda}a))j_{G}(s)&=&j_{\B}((\beta_{r}\circ\phi_{e})(a_{\lambda}a))j_{G}(s)\\
&=&j_{G}(r)j_{\B}(\phi_{e}(a_{\lambda}a))j_{G}(r)^{*}j_{G}(s)\\
&=&j_{G}(r)j_{\B}(\phi_{e}(a_{\lambda})\overline{\phi}_{e}(1)\phi_{e}(a))j_{G}(r^{-1})j_{G}(s)\\
&=&j_{G}(r)j_{\B}(\phi_{e}(a_{\lambda}))\overline{j_{\B}}(\overline{\phi}_{e}(1))j_{\B}(\phi_{e}(a))j_{G}(r^{-1}s)\\
&=&[j_{G}(r)j_{\B}(\phi_{e}(a_{\lambda}))]p[j_{\B}(\phi_{e}(a))j_{G}(r^{-1}s)].\\
\end{array}
\end{eqnarray*}
As the bottom line belongs to $(\B\rtimes_{\beta} G)p(\B\rtimes_{\beta} G)$, so does $j_{\B}(\phi_{r}(a_{\lambda}a))j_{G}(s)$. It
therefore follows that
$j_{\B}(\phi_{r}(a))j_{G}(s)\in\overline{(\B\rtimes_{\beta} G)p(\B\rtimes_{\beta} G)}$, which implies that
$$\overline{(\B\rtimes_{\beta} G)p(\B\rtimes_{\beta} G)}=\B\rtimes_{\beta} G.$$
Thus, $\T_{\textrm{cov}}(A\times_{\alpha} P)$ is a full corner in $\B\rtimes_{\beta} G$. This completes the proof.
\end{proof}

\begin{lemma}
\label{ker}
The ideal
$$\I=\ker \Omega=\clsp\{i_{P}(x)^{*}i_{A}(a)(1-i_{P}(s)^{*}i_{P}(s))i_{P}(y): a\in A, x,y,s\in P\}$$
of $\T_{\textrm{cov}}(A\times_{\alpha} P)$ is isomorphic to $p(\J\rtimes_{\beta} G)p$, which is a
full corner in $\J\rtimes_{\beta} G$.
\end{lemma}

\begin{proof}
We show that the isomorphism $\Psi$ in Theorem \ref{main} maps the ideal $\I$ onto $p(\J\rtimes_{\beta} G)p$. Firstly,
note that, as $\J\rtimes_{\beta} G$ sits in $\B\rtimes_{\beta} G$ as an essential ideal, $\B\rtimes_{\beta} G$
is embedded in $\M(\J\rtimes_{\beta} G)$ as a $C^{*}$-subalgebra as well as $\M(\B\rtimes_{\beta} G)$. Therefore,
$p\in \M(\J\rtimes_{\beta} G)$. Also recall that $\J\rtimes_{\beta} G$ is spanned by elements of the form
$j_{\B}(\phi_{r}(a)-\phi_{s}(\alpha_{sr^{-1}}(a)))j_{G}(z)$, where $a\in A$ and $r,s,z\in G$ such that $r<s$. Now, in order to see that $\Psi(\I)=p(\J\rtimes_{\beta} G)p$, let $a\in A$ and $x,y,t\in P$. We have
$$\Psi(i_{P}(x)^{*}i_{A}(a)(1-i_{P}(t)^{*}i_{P}(t))i_{P}(y))=k_{P}(x)^{*}k_{A}(a)(p-k_{P}(t)^{*}k_{P}(t))k_{P}(y).$$ Then,
by a similar calculation to the one that implies (\ref{inject}),
we have $$k_{A}(a)(p-k_{P}(t)^{*}k_{P}(t))=j_{\B}(\phi_{e}(a)-\phi_{t}(\alpha_{t}(a))),$$ and hence, we obtain
\begin{align}
\label{eq2}
k_{P}(x)^{*}k_{A}(a)(p-k_{P}(t)^{*}k_{P}(t))k_{P}(y)=p[j_{G}(x)j_{\B}(\phi_{e}(a)-\phi_{t}(\alpha_{t}(a)))j_{G}(y)^{*}]p
\end{align}
which belongs to $p(\J\rtimes_{\beta} G)p$. So it follows that $\Psi(\I)\subset p(\J\rtimes_{\beta} G)p$. For the other inclusion, we show that each spanning element
$p[j_{\B}(\phi_{r}(a)-\phi_{s}(\alpha_{sr^{-1}}(a)))j_{G}(z)]p$ of $p(\J\rtimes_{\beta} G)p$ belongs to $\Psi(\I)$. In order to do so, first, we can assume that $r,s\in P$ with $r<s$ without loss of generality.
This is due to the fact that
\begin{eqnarray*}
\begin{array}{l}
p[j_{\B}(\phi_{r}(a)-\phi_{s}(\alpha_{sr^{-1}}(a)))j_{G}(z)]p\\
=p[\overline{j_{\B}}(\overline{\phi}_{e}(1))j_{\B}(\phi_{r}(a)-\phi_{s}(\alpha_{sr^{-1}}(a)))j_{G}(z)]p\\
=p[j_{\B}(\overline{\phi}_{e}(1)\phi_{r}(a)-\overline{\phi}_{e}(1)\phi_{s}(\alpha_{sr^{-1}}(a)))j_{G}(z)]p,
\end{array}
\end{eqnarray*}
where
\begin{eqnarray*}
\begin{array}{l}
\overline{\phi}_{e}(1)\phi_{r}(a)-\overline{\phi}_{e}(1)\phi_{s}(\alpha_{sr^{-1}}(a))\\
=\phi_{e\vee r}(\overline{\alpha}_{e\vee r}(1)\alpha_{(e\vee r)r^{-1}}(a))-\phi_{e\vee s}(\overline{\alpha}_{e\vee s}(1)\alpha_{(e\vee s)s^{-1}}(\alpha_{sr^{-1}}(a)))\\
=\phi_{e\vee r}(\overline{\alpha}_{e\vee r}(1)\alpha_{(e\vee r)r^{-1}}(a))-\phi_{e\vee s}(\overline{\alpha}_{e\vee s}(1)\alpha_{(e\vee s)r^{-1}}(a)).
\end{array}
\end{eqnarray*}
Thus, since $r<s$, $e\leq e\vee r\leq e\vee s$, from which, for $\overline{\alpha}_{e\vee s}(1)\alpha_{(e\vee s)r^{-1}}(a)$, we have
$$\alpha_{(e\vee s)(e\vee r)^{-1}}(\overline{\alpha}_{e\vee r}(1)\alpha_{(e\vee r)r^{-1}}(a))=\overline{\alpha}_{e\vee s}(1)\alpha_{(e\vee s)r^{-1}}(a).$$
Consequently, we see that each spanning element $p[j_{\B}(\phi_{r}(a)-\phi_{s}(\alpha_{sr^{-1}}(a)))j_{G}(z)]p$ of $p(\J\rtimes_{\beta} G)p$ can actually be written in the form
$$p[j_{\B}(\phi_{x}(b)-\phi_{y}(\alpha_{yx^{-1}}(b)))j_{G}(z)]p$$ for some $b\in A$ and $x,y\in P$ with $x<y$. So, take any spanning element $p[j_{\B}(\phi_{r}(a)-\phi_{s}(\alpha_{sr^{-1}}(a)))j_{G}(z)]p$
of $p(\J\rtimes_{\beta} G)p$, where $a\in A$, $z\in G$, and $r,s\in P$ with $r<s$. It follows by the covariance equation of $(j_{\B},j_{G})$ that
\begin{eqnarray*}
\begin{array}{l}
p[j_{\B}(\phi_{r}(a)-\phi_{s}(\alpha_{sr^{-1}}(a)))j_{G}(z)]p\\
=p[j_{G}(z)(j_{\B}\circ\beta_{z^{-1}})(\phi_{r}(a)-\phi_{s}(\alpha_{sr^{-1}}(a)))]p\\
=p[j_{G}(z)j_{\B}(\phi_{z^{-1}r}(a)-\phi_{z^{-1}s}(\alpha_{sr^{-1}}(a)))\overline{j_{\B}}(\overline{\phi}_{e}(1))]p\\
=p[j_{G}(z)j_{\B}(\phi_{z^{-1}r}(a)\overline{\phi}_{e}(1)-\phi_{z^{-1}s}(\alpha_{sr^{-1}}(a))\overline{\phi}_{e}(1))]p,
\end{array}
\end{eqnarray*}
where
\begin{eqnarray*}
\begin{array}{l}
\phi_{z^{-1}r}(a)\overline{\phi}_{e}(1)-\phi_{z^{-1}s}(\alpha_{sr^{-1}}(a))\overline{\phi}_{e}(1)\\
=\phi_{(z^{-1}r)\vee e}(\alpha_{(z^{-1}r\vee e)r^{-1}z}(a)\overline{\alpha}_{(z^{-1}r\vee e)}(1))-\phi_{(z^{-1}s)\vee e}(\alpha_{(z^{-1}s\vee e)s^{-1}z}(\alpha_{sr^{-1}}(a))\overline{\alpha}_{(z^{-1}s\vee e)}(1))\\
=\phi_{(z^{-1}r)\vee e}(\alpha_{(z^{-1}r\vee e)r^{-1}z}(a)\overline{\alpha}_{(z^{-1}r\vee e)}(1))-\phi_{(z^{-1}s)\vee e}(\alpha_{(z^{-1}s\vee e)zr^{-1}}(a)\overline{\alpha}_{(z^{-1}s\vee e)}(1)).
\end{array}
\end{eqnarray*}
So, as $r<s$, $z^{-1}r<z^{-1}s$, and hence, $e\leq (z^{-1}r)\vee e\leq (z^{-1}s)\vee e$. Thus, for $\alpha_{(z^{-1}s\vee e)zr^{-1}}(a)\overline{\alpha}_{(z^{-1}s\vee e)}(1)$, we can write
\begin{eqnarray*}
\begin{array}{rcl}
\alpha_{(z^{-1}s\vee e)(z^{-1}r\vee e)^{-1}}(\alpha_{(z^{-1}r\vee e)r^{-1}z}(a)\overline{\alpha}_{(z^{-1}r\vee e)}(1))&=&\alpha_{(z^{-1}s\vee e)r^{-1}z}(a)\overline{\alpha}_{(z^{-1}s\vee e)}(1)\\
&=&\alpha_{(z^{-1}s\vee e)zr^{-1}}(a)\overline{\alpha}_{(z^{-1}s\vee e)}(1).
\end{array}
\end{eqnarray*}
Therefore, it follows that
$$p[j_{\B}(\phi_{r}(a)-\phi_{s}(\alpha_{sr^{-1}}(a)))j_{G}(z)]p=p[j_{G}(z)j_{\B}(\phi_{x}(b)-\phi_{y}(\alpha_{yx^{-1}}(b)))]p,$$ where
$x=(z^{-1}r)\vee e$, $y=(z^{-1}s)\vee e$, and $b=\alpha_{(z^{-1}r\vee e)r^{-1}z}(a)\overline{\alpha}_{(z^{-1}r\vee e)}(1)$. Then,
\begin{eqnarray*}
\begin{array}{l}
p[j_{G}(z)j_{\B}(\phi_{x}(b)-\phi_{y}(\alpha_{yx^{-1}}(b)))]p\\
=p[j_{G}(z)(j_{\B}\circ\beta_{x})(\phi_{e}(b)-\phi_{yx^{-1}}(\alpha_{yx^{-1}}(b)))]p\\
=p[j_{G}(z)j_{G}(x)j_{\B}(\phi_{e}(b)-\phi_{yx^{-1}}(\alpha_{yx^{-1}}(b)))j_{G}(x)^{*}]p\\
=p[j_{G}(zx)j_{\B}(\phi_{e}(b)-\phi_{yx^{-1}}(\alpha_{yx^{-1}}(b)))j_{G}(x)^{*}]p.
\end{array}
\end{eqnarray*}
Now, $x=(z^{-1}r)\vee e\in P$, clearly, and for $zx$, we have
$$e\leq r=z(z^{-1}r)\leq z((z^{-1}r)\vee e)=zx.$$ Thus, $zx\in P$, too, and therefore,
\begin{eqnarray*}
\begin{array}{l}
p[j_{G}(zx)j_{\B}(\phi_{e}(b)-\phi_{yx^{-1}}(\alpha_{yx^{-1}}(b)))j_{G}(x)^{*}]p\\
=k_{P}(zx)^{*}k_{A}(b)(p-k_{P}(yx^{-1})^{*}k_{P}(yx^{-1}))k_{P}(x)\ \ [\textrm{see}\ (\ref{eq2})]\\
=\Psi(i_{P}(zx)^{*}i_{A}(b)(1-i_{P}(yx^{-1})^{*}i_{P}(yx^{-1}))i_{P}(x))\in \Psi(\I).
\end{array}
\end{eqnarray*}
So, $p[j_{\B}(\phi_{r}(a)-\phi_{s}(\alpha_{sr^{-1}}(a)))j_{G}(z)]p$ belongs to $\Psi(\I)$ which implies that $p(\J\rtimes_{\beta} G)p\subset \Psi(\I)$. Therefore, we indeed have
$$\I\simeq \Psi(\I)=p(\J\rtimes_{\beta} G)p.$$

In order to show that $\I$ is a full corner in $\J\rtimes_{\beta} G$, we take an approximate identity in $\{a_{\lambda}\}$ in $A$. Then, for any spanning element spanned
$j_{\B}(\phi_{r}(a)-\phi_{s}(\alpha_{sr^{-1}}(a)))j_{G}(z)$ of $\J\rtimes_{\beta} G$, where $a\in A$ and $r,s,z\in G$ with $r<s$, we have
$$j_{\B}(\phi_{r}(aa_{\lambda})-\phi_{s}(\alpha_{sr^{-1}}(aa_{\lambda})))j_{G}(z)\rightarrow j_{\B}(\phi_{r}(a)-\phi_{s}(\alpha_{sr^{-1}}(a)))j_{G}(z)$$ in $\J\rtimes_{\beta} G$ with norm topology.
Now, for $j_{\B}(\phi_{r}(aa_{\lambda})-\phi_{s}(\alpha_{sr^{-1}}(aa_{\lambda})))j_{G}(z)$, by applying the covariance equation of $(j_{\B},j_{G})$, we have
\begin{eqnarray*}
\begin{array}{l}
j_{\B}(\phi_{r}(aa_{\lambda})-\phi_{s}(\alpha_{sr^{-1}}(aa_{\lambda})))j_{G}(z)\\
=(j_{\B}\circ\beta_{r})(\phi_{e}(aa_{\lambda})-\phi_{sr^{-1}}(\alpha_{sr^{-1}}(aa_{\lambda})))j_{G}(z)\\
=j_{G}(r)j_{\B}(\phi_{e}(aa_{\lambda})-\phi_{sr^{-1}}(\alpha_{sr^{-1}}(aa_{\lambda})))j_{G}(r)^{*}j_{\Gamma}(z)\\
=j_{G}(r)j_{\B}(\phi_{e}(aa_{\lambda})-\phi_{sr^{-1}}(\alpha_{sr^{-1}}(aa_{\lambda})))j_{G}(r^{-1})j_{\Gamma}(z)\\
=j_{G}(r)j_{\B}(\phi_{e}(aa_{\lambda})-\phi_{sr^{-1}}(\alpha_{sr^{-1}}(aa_{\lambda})))j_{G}(r^{-1}z).\\
\end{array}
\end{eqnarray*}
Then, in the bottom line, for $\phi_{e}(aa_{\lambda})-\phi_{sr^{-1}}(\alpha_{sr^{-1}}(aa_{\lambda}))$, it follows by the equation (\ref{eq3}) that
\begin{eqnarray*}
\begin{array}{l}
j_{G}(r)j_{\B}(\phi_{e}(aa_{\lambda})-\phi_{sr^{-1}}(\alpha_{sr^{-1}}(aa_{\lambda})))j_{G}(r^{-1}z)\\
=j_{G}(r)j_{\B}([\phi_{e}(a)-\phi_{sr^{-1}}(\alpha_{sr^{-1}}(a))][\phi_{e}(a_{\lambda})-\phi_{sr^{-1}}(\alpha_{sr^{-1}}(a_{\lambda}))])j_{G}(r^{-1}z)\\
=j_{G}(r)j_{\B}(\phi_{e}(a)-\phi_{sr^{-1}}(\alpha_{sr^{-1}}(a)))j_{\B}(\phi_{e}(a_{\lambda})-\phi_{sr^{-1}}(\alpha_{sr^{-1}}(a_{\lambda})))j_{G}(r^{-1}z)\\
=j_{G}(r)j_{\B}([\phi_{e}(a)-\phi_{sr^{-1}}(\alpha_{sr^{-1}}(a))]\overline{\phi}_{e}(1))
j_{\B}(\overline{\phi}_{e}(1)[\phi_{e}(a_{\lambda})-\phi_{sr^{-1}}(\alpha_{sr^{-1}}(a_{\lambda}))])j_{G}(r^{-1}z)\\
=j_{G}(r)j_{\B}(\phi_{e}(a)-\phi_{sr^{-1}}(\alpha_{sr^{-1}}(a)))\overline{j_{\B}}(\overline{\phi}_{e}(1))j_{\B}(\phi_{e}(a_{\lambda})-\phi_{sr^{-1}}(\alpha_{sr^{-1}}(a_{\lambda})))j_{G}(r^{-1}z)\\
=[j_{G}(r)j_{\B}(\phi_{e}(a)-\phi_{sr^{-1}}(\alpha_{sr^{-1}}(a)))]p[j_{\B}(\phi_{e}(a_{\lambda})-\phi_{sr^{-1}}(\alpha_{sr^{-1}}(a_{\lambda})))j_{G}(r^{-1}z)],
\end{array}
\end{eqnarray*}
which is an element of $(\J\rtimes_{\beta} G)p(\J\rtimes_{\beta} G)$. So, it follows that $$j_{\B}(\phi_{r}(a)-\phi_{s}(\alpha_{sr^{-1}}(a)))j_{G}(z)\in\overline{(\J\rtimes_{\beta} G)p(\J\rtimes_{\beta} G)},$$ and hence
$\overline{(\J\rtimes_{\beta} G)p(\J\rtimes_{\beta} G)}=\J\rtimes_{\beta} G$. Therefore, $\I\simeq p(\J\rtimes_{\beta} G)p$ is a full corner in $\J\rtimes_{\beta} G$.
\end{proof}

\begin{prop}
\label{essen}
The ideal $\I$ is an essential ideal of the Nica-Toeplitz algebra $\T_{\textrm{cov}}(A\times_{\alpha} P)$.
\end{prop}

\begin{proof}
This is due to the facts that the ideal $\I$ and $\T_{\textrm{cov}}(A\times_{\alpha} P)$ are full corners in $\J\rtimes_{\beta} G$ and $\B\rtimes_{\beta} G$, respectively, and $\J\rtimes_{\beta} G$ is an essential ideal of $\B\rtimes_{\beta} G$.
\end{proof}

\begin{remark}
\label{(Z,N)}
We would like to mention that when $(G,P)$ is an abelian totally ordered group we completely fall into the context of \cite{SZ}, as $\T_{\textrm{cov}}(A\times_{\alpha} P)\simeq A\times_{\alpha}^{\piso} P$ (see \cite{LR}).
\end{remark}

\section{The Nica-Toeplitz algebra $\T_{\textrm{cov}}(A\times_{\alpha} P)$ of dynamical systems by semigroups of automorphisms}
\label{sec:auto}
Let $(G,P)$ be an abelian lattice-ordered group. In this section we suppose that $(A,P,\alpha)$ is a system consisting
of a $C^{*}$-algebra $A$ and an action $\alpha:P\rightarrow \Aut(A)$ of $P$ by automorphisms of $A$. First, we recall that
the automorphic action $\alpha$ of $P$ can be extended (uniquely) to an action of the group $G$ on $A$ by using the fact
that $G=P^{-1}P$. More precisely, for every $s\in G$, $s\vee e\in P$, and since $(s\vee e)s^{-1}\in P$, $s(s\vee e)^{-1}\in P^{-1}$. Therefore,
$$s=(s(s\vee e)^{-1})(s\vee e)=(s\wedge e)(s\vee e)\in P^{-1}P.$$
Now, for every $t\in G$ with $t=r^{-1}s$, where $r,s \in P$, if we define $\alpha_{t}:=\alpha_{r}^{-1}\alpha_{s}$, then we obtain an
action of $G$ on $A$ by automorphisms, which is the extension of $\alpha$ to $G$.
Hence we get a dynamical system $(A,G,\alpha)$. We actually have the following lemma.

\begin{lemma}
\label{alpha extension}
Suppose that $(A,P,\alpha)$ is a system consisting of a $C^{*}$-algebra $A$ and an action $\alpha:P\rightarrow \Aut(A)$ of $P$ by automorphisms of $A$. Then, the action $\alpha$ extends uniquely to an action of $G$ on $A$ by automorphisms.
\end{lemma}

\begin{proof}
We skip the proof and readers are referred to \cite[Theorem 1.2]{Marcelo} for more details.
\end{proof}

Now let $B_{G}$ be the $C^{*}$-subalgebra of $\ell^{\infty}(G)$ generated by the characteristic functions $\{1_{s}\in\ell^{\infty}(G):s\in G\}$, such that
\[
1_{s}(t)=
   \begin{cases}
      1 &\textrm{if}\empty\ \text{$s\leq t$}\\
      0 &\textrm{otherwise}.
   \end{cases}
\]
It can easily be seen that $1_{s}1_{t}=1_{s\vee t}$, and $1_{s}^{*}=1_{s}$ for every $s,t\in G$. Thus, we have $$B_{G}=\clsp\{1_{s}: s\in G\}.$$
Let $\tau:G\rightarrow \Aut(B_{G})$ be the action of $G$ on $B_{G}$ by automorphisms given by the translation, such that $\tau_{t}(1_{s})=1_{ts}$ for all $s,t\in G$. Hence, we have the dynamical system $(B_{G},G,\tau)$.
Let $B_{G,\infty}$ denote the $C^{*}$-subalgebra of $B_{G}$ generated by the elements $\{1_{s}-1_{t}: s<t\in G\}$. It is not difficult to see that
$$B_{G,\infty}=\clsp\{1_{s}-1_{t}: s<t\in G\},$$ which is an essential ideal of $B_{G}$. It is $\tau$-invariant, too. Moreover, $s\mapsto \tau_{s}\otimes \alpha_{s}^{-1}$
defines an action of $G$ on the algebra $B_{G}\otimes A$ by automorphisms (note that as the algebra $B_{G}$ is abelian, $B_{G}\otimes A=B_{G}\otimes_{\textrm{min}} A=B_{G}\otimes_{\textrm{max}} A$).
Therefore, we obtain a dynamical system $(B_{G}\otimes A, G, \tau\otimes \alpha^{-1})$. Also, $B_{G,\infty}\otimes A$ is a
$(\tau\otimes \alpha^{-1})$-invariant ideal of $B_{G}\otimes A$. Next, we see that the algebra $\B$ and its ideal $\J$ associated with the system $(A,P,\alpha)$ can be identified with tensor product algebras.

\begin{prop}
\label{B-auto}
There is an isomorphism $\mu:B_{G}\otimes A\rightarrow \B$ such that $\beta_{t}(\mu(\xi))=\mu((\tau\otimes\alpha^{-1})_{t}(\xi))$ for all $\xi\in(B_{G}\otimes A)$ and $t\in G$,
and it maps the ideal $B_{G,\infty}\otimes A$ onto $\J$. Moreover, $\mu$ induces an isomorphism $\Gamma$ of $((B_{G}\otimes A)\rtimes_{\tau\otimes\alpha^{-1}} G,i)$ onto $(\B\rtimes_{\beta} G,j)$ such that
\begin{align}
\label{eq5}
\Gamma(i_{B_{G}\otimes A}(\xi)i_{G}(s))=j_{\B}(\mu(\xi))j_{G}(s)\ \ \textrm{for all}\ \xi\in(B_{G}\otimes A), s\in G,
\end{align}
and it maps the ideal $(B_{G,\infty}\otimes A)\rtimes_{\tau\otimes\alpha^{-1}} G$ onto $\J\rtimes_{\beta} G$.
\end{prop}

\begin{proof}
Firstly, since $\ell^{\infty}(G,A)$ sits in $\ell^{\infty}(G,\M(A))$ as an essential ideal, $\ell^{\infty}(G,\M(A))$ is embedded in $\M(\ell^{\infty}(G,A))$ as a  $C^*$-subalgebra. Now, define the maps
$$\varphi:B_{G}\rightarrow \M(\ell^{\infty}(G,A))\ \ \textrm{and}\ \ \psi:A\rightarrow \M(\ell^{\infty}(G,A))$$
by $$\varphi(f)(s)=f(s)1_{\M(A)}\ \ \textrm{and}\ \ \psi(a)(s)=\alpha_{s}(a)$$ for every $f\in B_{G}$, $a\in A$, and $s\in G$. One can see that $\varphi$ and $\psi$ are $*$-homomorphisms  with commuting ranges, which means that
$\varphi(f)\psi(a)=\psi(a)\varphi(f)$ for all $f\in B_{G}$ and $a\in A$. Therefore, there exists a homomorphism $\mu:=\varphi\otimes\psi:B_{G}\otimes A\rightarrow \M(\ell^{\infty}(G,A))$ such that
$\mu(f\otimes a)=\varphi(f)\psi(a)=\psi(a)\varphi(f)$ for every $f\in B_{G}$ and $a\in A$. We prove that $\mu$ is actually an isomorphism of $B_{G}\otimes A$ onto the algebra $\B$. To do so, we first show that $\mu(B_{G}\otimes A)=\B$.
For each spanning element $1_{s}\otimes a$ of $B_{G}\otimes A$, we have
\[
\mu(1_{s}\otimes a)(t)=(\varphi(1_{s})\psi(a))(t)=\varphi(1_{s})(t)\psi(a)(t)=1_{s}(t)\alpha_{t}(a)=
   \begin{cases}
      \alpha_{t}(a) &\textrm{if}\empty\ \text{$s\leq t$}\\
      0 &\textrm{otherwise},
   \end{cases}
\]
which is equal to
\[
\phi_{s}(\alpha_{s}(a))(t)=
   \begin{cases}
      \alpha_{ts^{-1}}(\alpha_{s}(a))=\alpha_{t}(a) &\textrm{if}\empty\ \text{$s\leq t$}\\
      0 &\textrm{otherwise},
   \end{cases}
\]
for every $t\in G$. So, we have $\mu(1_{s}\otimes a)=\phi_{s}(\alpha_{s}(a))\in\B$, and therefore $\mu(B_{G}\otimes A)\subset\B$. To see the other inclusion, for any spanning element $\phi_{s}(a)$ of $\B$, we apply the equation
$\mu(1_{s}\otimes a)=\phi_{s}(\alpha_{s}(a))$ to see that
$$\mu(1_{s}\otimes \alpha_{s^{-1}}(a))=\phi_{s}(\alpha_{s}(\alpha_{s^{-1}}(a)))=\phi_{s}(a).$$
Therefore, $\phi_{s}(a)=\mu(1_{s}\otimes \alpha_{s^{-1}}(a))\in\mu(B_{G}\otimes A)$, which implies that $\B\subset\mu(B_{G}\otimes A)$.

Next, we show that $\mu$ is injective. Define the map $M:B_{G}\rightarrow B(\ell^{2}(G))$ by $(M(f)\lambda)(s)=f(s)\lambda(s)$ for every $f\in B_{G}$ and $\lambda\in\ell^{2}(G)$, which is a faithful (nondegenerate) representation. Let
$\pi:A\rightarrow B(H)$ be a faithful (nondegenerate) representation of $A$ on some Hilbert space $H$. Then, it follows by \cite[Corollary B.11]{RW} that there is a faithful representation $M\otimes\pi:B_{G}\otimes A\rightarrow B(\ell^{2}(G)\otimes H)$ such that
$M\otimes\pi(f\otimes a)=M(f)\otimes \pi(a)$. On the other hand, we have a faithful representation $\tilde{\pi}:\B\rightarrow B(\ell^{2}(G,H))$ of
$\B$ on the Hilbert space $\ell^{2}(G,H)$ defined by $(\tilde{\pi}(\xi)\eta)(s)=\pi(\alpha_{s^{-1}}(\xi(s)))\eta(s)$ for every $\xi\in\B$ and $\eta\in\ell^{2}(G,H)$. Now, take $U$ to be the
isomorphism (unitary) of $\ell^{2}(G)\otimes H$ onto $\ell^{2}(G,H)$ which satisfies $U(\lambda\otimes h)(s)=\lambda(s)h$ for all $\lambda\in\ell^{2}(G)$ and $h\in H$. So, we have
\begin{eqnarray*}
\begin{array}{rcl}
\big(\tilde{\pi}(\mu(f\otimes a))U(\lambda\otimes h)\big)(s)&=&\pi(\alpha_{s^{-1}}(\mu(f\otimes a)(s)))U(\lambda\otimes h)(s)\\
&=&\pi(\alpha_{s^{-1}}(f(s)\alpha_{s}(a)))(\lambda(s)h)\\
&=&\pi(f(s)a)(\lambda(s)h)\\
&=&f(s)\lambda(s)\pi(a)h,
\end{array}
\end{eqnarray*}
and
\begin{eqnarray*}
\begin{array}{rcl}
U\big((M\otimes\pi(f\otimes a))(\lambda\otimes h)\big)(s)&=&U\big((M(f)\otimes \pi(a))(\lambda\otimes h)\big)(s)\\
&=&U\big(M(f)\lambda\otimes \pi(a)h)(s)\\
&=&(M(f)\lambda)(s)\pi(a)h=f(s)\lambda(s)\pi(a)h.
\end{array}
\end{eqnarray*}
Therefore, we have $$\tilde{\pi}(\mu(f\otimes a))U(\lambda\otimes h)=U\big((M\otimes\pi(f\otimes a))(\lambda\otimes h)\big),$$ which implies that
$$U^{*}\tilde{\pi}(\mu(\xi))U=(M\otimes\pi)(\xi)$$ for all $\xi\in B_{G}\otimes A$. So, it follows that $\mu$ must be injective. This is due to the facts that $\tilde{\pi}$ and $M\otimes\pi$ are injective, and $U$ is a unitary.
Consequently, $B_{G}\otimes A\simeq\mu(B_{G}\otimes A)=\B$. Moreover, $B_{G,\infty}\otimes A$ is isomorphic to
$\J$ via $\mu$. To see this, take $a\in A$ and $s<t\in G$. Then,
\begin{eqnarray}
\label{eq4}
\begin{array}{rcl}
\mu((1_{s}-1_{t})\otimes a)&=&\mu((1_{s}\otimes a)-(1_{t}\otimes a))\\
&=&\mu(1_{s}\otimes a)-\mu(1_{t}\otimes a)\\
&=&\phi_{s}(\alpha_{s}(a))-\phi_{t}(\alpha_{t}(a))\\
&=&\phi_{s}(\alpha_{s}(a))-\phi_{t}(\alpha_{ts^{-1}}(\alpha_{s}(a)))\in\J.
\end{array}
\end{eqnarray}
Therefore, $\mu(B_{G,\infty}\otimes A)\subset\J$. For the other inclusion, by the above computation in (\ref{eq4}), we have
\begin{eqnarray*}
\begin{array}{rcl}
\mu((1_{s}-1_{t})\otimes \alpha_{s^{-1}}(a))&=&\phi_{s}(\alpha_{s}(\alpha_{s^{-1}}(a)))-\phi_{t}(\alpha_{t}(\alpha_{s^{-1}}(a)))\\
&=&\phi_{s}(a)-\phi_{t}(\alpha_{ts^{-1}}(a)).
\end{array}
\end{eqnarray*}
So, each spanning element $\phi_{s}(a)-\phi_{t}(\alpha_{ts^{-1}}(a))$ of $\J$ is equal to $\mu((1_{s}-1_{t})\otimes \alpha_{s^{-1}}(a))$, which belongs to $\mu(B_{G,\infty}\otimes A)$. Therefore, $\J\subset\mu(B_{G,\infty}\otimes A)$, and hence
$\mu(B_{G,\infty}\otimes A)=\J$. This means that $B_{G,\infty}\otimes A\simeq \J$ via $\mu$.

At last, we show that the isomorphism $\mu$ satisfies $\beta_{t}\circ\mu=\mu\circ(\tau\otimes\alpha^{-1})_{t}$. Therefore, by \cite[Lemma 2.65]{W}, there is an isomorphism
$\Gamma:((B_{G}\otimes A)\rtimes_{\tau\otimes\alpha^{-1}} G,i)\rightarrow (\B\rtimes_{\beta} G,j)$ such that
$$\Gamma(i_{B_{G}\otimes A}(\xi)i_{G}(s))=j_{\B}(\mu(\xi))j_{G}(s)\ \ \textrm{for all}\ \xi\in(B_{G}\otimes A), s\in G.$$
For each spanning element $1_{s}\otimes a$ of $B_{G}\otimes A$, we have
$$\beta_{t}(\mu(1_{s}\otimes a))=\beta_{t}(\phi_{s}(\alpha_{s}(a)))=\phi_{ts}(\alpha_{s}(a)).$$
On the other hand,
\begin{eqnarray*}
\begin{array}{rcl}
\mu((\tau\otimes\alpha^{-1})_{t}(1_{s}\otimes a))&=&\mu(\tau_{t}\otimes\alpha^{-1}_{t}(1_{s}\otimes a))\\
&=&\mu(\tau_{t}(1_{s})\otimes\alpha^{-1}_{t}(a))\\
&=&\mu(1_{ts}\otimes\alpha_{t^{-1}}(a))\\
&=&\phi_{ts}(\alpha_{ts}(\alpha_{t^{-1}}(a)))\ \ [\textrm{by applying}\ \mu(1_{s}\otimes a)=\phi_{s}(\alpha_{s}(a))]\\
&=&\phi_{ts}(\alpha_{s}(a)).\\
\end{array}
\end{eqnarray*}
Thus, $\beta_{t}\circ\mu=\mu\circ(\tau\otimes\alpha^{-1})_{t}$ is valid. Note that, by some routine computation on spanning elements using the equation (\ref{eq5}), it follows that
$$(B_{G,\infty}\otimes A)\rtimes_{\tau\otimes\alpha^{-1}} G\simeq \Gamma\big((B_{G,\infty}\otimes A)\rtimes_{\tau\otimes\alpha^{-1}} G\big)=\J\rtimes_{\beta} G.$$ We skip it here.
\end{proof}

\begin{cor}
\label{full NT-auto}
If $q=\overline{i}_{B_{G}\otimes A}(1_{e}\otimes 1_{\M(A)})\in\M\big((B_{G}\otimes A)\rtimes_{\tau\otimes\alpha^{-1}} G\big)$, then $\overline{\Gamma}(q)=p$. Thus, it follows that
$\T_{\textrm{cov}}(A\times_{\alpha} P)$ and the ideal $\I$ are isomorphic to the full corners $q[(B_{G}\otimes A)\rtimes_{\tau\otimes\alpha^{-1}} G]q$ and $q[(B_{G,\infty}\otimes A)\rtimes_{\tau\otimes\alpha^{-1}} G]q$, respectively.
\end{cor}

\begin{proof}
First of all, as the homomorphism $j_{\B}$ is nondegenerate, so is the isomorphism $\Gamma$. Therefore, $\Gamma$ extends to an isometry of multiplier algebras. Now, take any approximate identity $\{a_{\lambda}\}$ in $A$. Then,
it follows by the equation (\ref{eq5}) that
$$\Gamma(i_{B_{G}\otimes A}(1_{e}\otimes a_{\lambda}))=j_{\B}(\mu(1_{e}\otimes a_{\lambda}))=j_{\B}(\phi_{e}(a_{\lambda})).$$
Thus, since $1_{e}\otimes a_{\lambda}\rightarrow 1_{e}\otimes 1_{\M(A)}$ strictly in $\M(B_{G}\otimes A)$, in the equation above, the left hand side tends to $\overline{\Gamma}(\overline{i}_{B_{G}\otimes A}(1_{e}\otimes 1_{\M(A)}))=\overline{\Gamma}(q)$,
while the right hand side tends to $\overline{j_{\B}}(\overline{\phi}_{e}(1))=p$ strictly in $\M(\B\rtimes_{\beta} G)$. Hence, we have $\overline{\Gamma}(q)=p$. Therefore, it follows by Proposition \ref{B-auto} that
$$q[(B_{G}\otimes A)\rtimes_{\tau\otimes\alpha^{-1}} G]q\simeq\Gamma\big(q[(B_{G}\otimes A)\rtimes_{\tau\otimes\alpha^{-1}} G]q\big)=p(\B\rtimes_{\beta} G)p,$$ and
$$q[(B_{G,\infty}\otimes A)\rtimes_{\tau\otimes\alpha^{-1}} G]q\simeq\Gamma\big(q[(B_{G,\infty}\otimes A)\rtimes_{\tau\otimes\alpha^{-1}} G]q\big)=p(\J\rtimes_{\beta} G)p,$$
where by Theorem \ref{main} and Lemma \ref{ker}, $p(\B\rtimes_{\beta} G)p$ and $p(\J\rtimes_{\beta} G)p$ are isomorphic to $\T_{\textrm{cov}}(A\times_{\alpha} P)$ and the ideal $\I$, respectively. Consequently,
$$\T_{\textrm{cov}}(A\times_{\alpha} P)\simeq q[(B_{G}\otimes A)\rtimes_{\tau\otimes\alpha^{-1}} G]q\ \ \textrm{and}
\ \ \I\simeq q[(B_{G,\infty}\otimes A)\rtimes_{\tau\otimes\alpha^{-1}} G]q$$ via the isomorphism $\Gamma^{-1}\circ \Psi$.
\end{proof}

Of course it is natural to ask that whether there is a familiar identification for the algebra $\T_{\textrm{cov}}(A\times_{\alpha} P)$ and its ideal $\I$ for the automorphic system
$(A,P,\alpha)$ with the trivial action $\alpha=\id$. To answer this question, we first need to recall
about the Toeplitz algebra $\T(P)$ briefly. For more, readers are referred to \cite{Nica} and \cite{LacaR}. Let $\{e_{x}: x\in P\}$ be the usual
orthonormal basis of the Hilbert space $\ell^{2}(P)$. There is a representation $T:P\rightarrow B(\ell^{2}(P))$ (called the \emph{Toeplitz representation}) of $P$ on $\ell^{2}(P)$ by isometries
such that $T_{x}(e_{y})=e_{xy}$, and
\begin{align}
\label{Nica-cov}
T_{x}^{*}T_{y}=T_{(x\vee y)x^{-1}}T_{(x\vee y)y^{-1}}^{*}
\end{align}
for all $x,y\in P$. Now, the Toeplitz algebra $\T(P)$ is the $C^{*}$-subalgebra of $B(\ell^{2}(P))$ generated by the isometries
$\{T_{x}: x\in P\}$. Moreover, since $(G,P)$ is abelian, $\T(P)$ is (isomorphic to) the isometric crossed product $B_{P}\times_{\tau}^{\iso} P$, where the action $\tau$ of $P$ on the algebra
$B_{P}=\clsp\{1_{y}: y\in P\}\subset \ell^{\infty}(P)$ is given by $\tau_{x}(1_{y})=1_{xy}$ for all $x,y\in P$. Thus, $\T(P)\simeq B_{P}\times_{\tau}^{\iso} P$ is indeed the universal $C^*$-algebra for isometric
representations $V:P\rightarrow B(H)$ of $P$ which are \emph{Nica covariant}, which means that they satisfy
\begin{align}
\label{Nica-cov2}
V_{x}^{*}V_{y}=V_{(x\vee y)x^{-1}}V_{(x\vee y)y^{-1}}^{*}
\end{align}
for all $x,y\in P$.

\begin{remark}
\label{(C-P-id)}
One can see that for the trivial system $(\C,P,\id)$ the corresponding dynamical system $(\B,G,\beta)$
(in Theorem \ref{main}) is nothing but the system $(B_{G},G,\tau)$. Therefore, by Theorem \ref{main}, the Nica-Toeplitz algebra
\begin{align}
\label{span-NT-id}
\T_{\textrm{cov}}(\C\times_{\id} P)=\overline{\newspan}\{i_{P}(x)^{*} i_{P}(y) : x,y \in P\}
\end{align}
of the system $(\C,P,\id)$ is isomorphic to the full corner $k_{B_{G}}(1_{e})(B_{G}\rtimes_{\tau} G)k_{B_{G}}(1_{e})$
of the crossed product $(B_{G}\rtimes_{\tau} G,k)$, such that for the isomorphism
$\Psi:\T_{\textrm{cov}}(\C\times_{\id} P)\rightarrow k_{B_{G}}(1_{e})(B_{G}\rtimes_{\tau} G)k_{B_{G}}(1_{e})$ we have $\Psi(i_{P}(x))=k_{B_{G}}(1_{e})k_{G}(x)^{*}k_{B_{G}}(1_{e})$ for all
$x\in P$. Note that here the projection $k_{B_{G}}(1_{e})$ in $B_{G}\rtimes_{\tau} G\subset \M(B_{G,\infty}\rtimes_{\tau} G)$
is the projection $p$. Moreover, the (essential) ideal $\J$ of $B_{G}$ is
the (essential) ideal $B_{G,\infty}$, and hence, the ideal
\begin{align}
\label{span-I-id}
\I=\clsp\{i_{P}(x)^{*}(1-i_{P}(s)^{*}i_{P}(s))i_{P}(y): x,y,s\in P\}
\end{align}
of the algebra $\T_{\textrm{cov}}(\C\times_{\id} P)$ is isomorphic to the full corner
$k_{B_{G}}(1_{e})(B_{G,\infty}\rtimes_{\tau} G)k_{B_{G}}(1_{e})$ of $B_{G,\infty}\rtimes_{\tau} G$
via the isomorphism $\Psi$ (see Lemma \ref{ker}), such that
\begin{align}
\label{eq21}
\Psi(i_{P}(x)^{*}(1-i_{P}(s)^{*}i_{P}(s))i_{P}(y))=k_{B_{G}}(1_{e})[k_{G}(x)k_{B_{G}}(1_{e}-1_{s})j_{G}(y)^{*}]k_{B_{G}}(1_{e}).
\end{align}
Now, in \cite{Li-2}, by using the dilation theory in \cite{Marcelo} for isometric crossed products, it was shown that
the Toeplitz algebra $\T(P)$ is isomorphic to the full corner $k_{B_{G}}(1_{e})(B_{G}\rtimes_{\tau} G)k_{B_{G}}(1_{e})$.
Therefore, it follows that the algebra $\T_{\textrm{cov}}(\C\times_{\id} P)$ is isomorphic to the Toeplitz algebra $\T(P)$.
Note that, this can be also seen directly by \cite[Corollary 7.11]{BLS}, and for the isomorphism
of $\T_{\textrm{cov}}(\C\times_{\id} P)$ onto $\T(P)$, which we denote it by $\Phi$, we have
$\Phi(i_{P}(x))=T_{x}^{*}$ for all $x\in P$.
\end{remark}

\begin{lemma}
\label{comm.I}
The ideal $\I\simeq k_{B_{G}}(1_{e})(B_{G,\infty}\rtimes_{\tau} G)k_{B_{G}}(1_{e})$ of $\T_{\textrm{cov}}(\C\times_{\id} P)$ is isomorphic the commutator ideal $\mathcal{C}_{P}$ of $\T(P)$, such that
we have
\begin{align}
\label{comm.I.span}
\mathcal{C}_{P}=\clsp\{T_{x}(1-T_{s}T_{s}^{*})T_{y}^{*}: x,y,s\in P\}.
\end{align}
Thus, (\ref{comm.I.span}) generalizes \cite[Lemma 2.4]{LR} to every lattice-ordered abelian group $(G,P)$.
\end{lemma}

\begin{proof}
By a simple calculation on the spanning elements of $\I$ (see (\ref{span-I-id})), one can see that it is isomorphic to the ideal
$$\widetilde{\I}=\clsp\{T_{x}(1-T_{s}T_{s}^{*})T_{y}^{*}: x,y,s\in P\}$$ of $\T(P)$ via the isomorphism $\Phi$. Then, since (see (\ref{exseq1}))
$$\frac{\T(P)}{\widetilde{\I}}\simeq \frac{\T_{\textrm{cov}}(\C\times_{\id} P)}{\I}\simeq \C\times_{\id}^{\iso} P\simeq \C\rtimes_{\id} G\simeq C^{*}(G)\simeq C(\widehat{G}),$$
in which $C(\widehat{G})$ is abelian, it follows that the commutator ideal $\mathcal{C}_{P}$ of $\T(P)$ must be contained in ${\widetilde{\I}}$. For other inclusion, as each
element $1-T_{s}T_{s}^{*}$ is actually the commutator $[T_{s}^{*},T_{s}]\in \mathcal{C}_{P}$, we have ${\widetilde{\I}}\subset \mathcal{C}_{P}$. Therefore,
${\widetilde{\I}}=\mathcal{C}_{P}$, and hence,
$$\I\simeq \Phi(\I)=\mathcal{C}_{P}=\clsp\{T_{x}(1-T_{s}T_{s}^{*})T_{y}^{*}: x,y,s\in P\}.$$
\end{proof}

\begin{remark}
\label{Murphy-exseq}
Note that, therefore, for the trivial system $(\C,P,\id)$, the short exact sequence (\ref{exseq1}) is the well-known exact sequence
\begin{align}
\label{Murphy-exseq2}
0 \longrightarrow \mathcal{C}_{P} \stackrel{}{\longrightarrow} \T(P) \stackrel{\psi}{\longrightarrow} C(\widehat{G}) \longrightarrow 0,
\end{align}
where $\psi$ is the surjective homomorphism which maps each $T_{x}$ to the evaluation map $\varepsilon_{x}$ (see \cite[Theorem 1.5]{murphy}).
Recall that the algebra $C(\widehat{G})$ is generated by the evaluation maps $\{\varepsilon_{x}: x\in P\}$ of
$\widehat{G}$ into $\TT\subset \C$ defined by $\varepsilon_{x}(\gamma)=\gamma(x)$ for all $\gamma\in \widehat{G}$.
\end{remark}

\begin{cor}
\label{NT-and-I-id}
Consider the system $(A,P,\alpha)$ with the trivial action $\alpha=\id$. Then, the Nica-Toeplitz algebra $\T_{\textrm{cov}}(A\times_{\alpha} P)$ and the ideal $\I$ corresponding to the system
are isomorphic to the tensor products $A\otimes_{\textrm{max}} \T(P)$ and $A\otimes_{\textrm{max}} \mathcal{C}_{P}$, respectively. Consequently, the short exact sequence (\ref{exseq1})
of the system is actually the exact sequence
\begin{align}
\label{tensor-exseq}
0 \longrightarrow A\otimes_{\textrm{max}} \mathcal{C}_{P} \stackrel{}{\longrightarrow} A\otimes_{\textrm{max}} \T(P) \stackrel{}{\longrightarrow} A\otimes_{\textrm{max}} C(\widehat{G}) \longrightarrow 0
\end{align}
obtained by the (maximal) tensor product with the $C^{*}$-algebra $A$ to (\ref{Murphy-exseq2}).
\end{cor}

\begin{proof}
Firstly, it is known that the crossed products of nuclear $C^{*}$-algebras by actions of amenable (locally compact) groups are nuclear. Therefore, since the algebra $B_{G}$ and the group $G$ are abelian,
the crossed product $B_{G}\rtimes_{\tau} G$ is nuclear. It thus follows that
$$(B_{G}\rtimes_{\tau} G)\otimes_{\textrm{max}} A=(B_{G}\rtimes_{\tau} G)\otimes_{\textrm{min}} A=(B_{G}\rtimes_{\tau} G)\otimes A.$$
Then, by \cite[Lemma 2.75]{W}, there is an isomorphism
$$\Delta: ((B_{G}\otimes A)\rtimes_{\tau\otimes\id} G,i)\rightarrow (B_{G}\rtimes_{\tau} G,k)\otimes_{\textrm{max}} A$$ such that
$$\Delta(i_{B_{G}\otimes A}(1_{t}\otimes a)i_{G}(s))=k_{B_{G}}(1_{t})k_{G}(s)\otimes a.$$ Since the homomorphism $k_{B_{G}}$ is nondegenerate, so is the isomorphism $\Delta$, and hence, it extends to
an isometry of multiplier algebras. Now, for any approximate identity $\{a_{\lambda}\}$ in $A$, we have
$$\Delta(i_{B_{G}\otimes A}(1_{e}\otimes a_{\lambda}))=k_{B_{G}}(1_{e})\otimes a_{\lambda}.$$ One can see that, in the equation above, the left hand side approaches
$\overline{\Delta}(\overline{i}_{B_{G}\otimes A}(1_{e}\otimes 1_{\M(A)}))=\overline{\Delta}(q)$, while the right hand side approaches $k_{B_{G}}(1_{e})\otimes 1_{\M(A)}$ strictly in
the multiplier algebra $\M((B_{G}\rtimes_{\tau} G,k)\otimes_{\textrm{max}} A)$. So, we must have $$\overline{\Delta}(q)=k_{B_{G}}(1_{e})\otimes 1_{\M(A)},$$ which is a projection in
$\M((B_{G}\rtimes_{\tau} G,k)\otimes_{\textrm{max}} A)$ and we denote it by $\tilde{q}$. We therefore have
\begin{eqnarray*}
\begin{array}{rcl}
\T_{\textrm{cov}}(A\times_{\id} P)&\simeq& q[(B_{G}\otimes A)\rtimes_{\tau\otimes\id} G]q\\
&\simeq&\Delta\big(q[(B_{G}\otimes A)\rtimes_{\tau\otimes\id} G]q\big)\\
&=&\tilde{q}[(B_{G}\rtimes_{\tau} G)\otimes_{\textrm{max}} A]\tilde{q}\\
&=&\tilde{q}[(B_{G}\rtimes_{\tau} G)\otimes_{\textrm{min}} A]\tilde{q}\\
&=&[k_{B_{G}}(1_{e})(B_{G}\rtimes_{\tau} G)k_{B_{G}}(1_{e})]\otimes_{\textrm{min}} A\\
&\simeq&\T_{\textrm{cov}}(\C \times_{\id} P)\otimes_{\textrm{min}} A \simeq \T(P)\otimes_{\textrm{min}} A,
\end{array}
\end{eqnarray*}
where $\T(P)\otimes_{\textrm{min}} A=\T(P)\otimes_{\textrm{max}} A=\T(P)\otimes A$. This is due to the fact that, since $G$ is abelian, by \cite[Corollary 6.45]{Li}, the algebra $\T(P)$
is indeed nuclear. So, more precisely, we have an isomorphism
$\Upsilon:(\T_{\textrm{cov}}(A\times_{\id} P),i_{A},i_{P})\rightarrow A\otimes_{\textrm{max}} \T(P)$ such that
$$\Upsilon(i_{P}(x)^{*}i_{A}(a)i_{P}(y))=a\otimes T_{x}T_{y}^{*}$$ for all $a\in A$ and $x,y\in P$ (one can see that when $A=\C$ the isomorphism $\Upsilon$ is just the isomorphism $\Phi$ in Remark \ref{(C-P-id)}).
In particular, $\Upsilon(i_{A}(a))=a\otimes 1_{\T(P)}$, and since $\Upsilon$ is nondegenerate (it is not difficult to see this as $A$ contains an approximate identity), we have
$\overline{\Upsilon}(i_{P}(x))=1_{\M(A)}\otimes T_{x}^{*}$. Moreover, $\Upsilon$ restricts to an isomorphism of the ideal $\I$ onto the ideal $A\otimes_{\textrm{max}} \mathcal{C}_{P}$ of
$A\otimes_{\textrm{max}} \T(P)$, as
\begin{eqnarray*}
\begin{array}{l}
\Upsilon(i_{P}(x)^{*}i_{A}(a)(1-i_{P}(s)^{*}i_{P}(s))i_{P}(y))\\
=\overline{\Upsilon}(i_{P}(x)^{*})\Upsilon(i_{A}(a))[\overline{\Upsilon}(1)-\overline{\Upsilon}(i_{P}(s)^{*})\overline{\Upsilon}(i_{P}(s))]\overline{\Upsilon}(i_{P}(y))\\
(1\otimes T_{x})(a\otimes 1)[1-(1\otimes T_{s})(1\otimes T_{s}^{*})](1\otimes T_{y}^{*})\\
(a\otimes T_{x})[1-(1\otimes T_{s}T_{s}^{*})](1\otimes T_{y}^{*})\\
(a\otimes T_{x})(1\otimes T_{y}^{*})-(a\otimes T_{x})(1\otimes T_{s}T_{s}^{*})(1\otimes T_{y}^{*})\\
(a\otimes T_{x}T_{y}^{*})-(a\otimes T_{x}T_{s}T_{s}^{*}T_{y}^{*})\\
a\otimes (T_{x}T_{y}^{*}-T_{x}T_{s}T_{s}^{*}T_{y}^{*})=a\otimes T_{x}(1-T_{s}T_{s}^{*})T_{y}^{*},
\end{array}
\end{eqnarray*}
which is a spanning element of $A\otimes_{\textrm{max}} \mathcal{C}_{P}$. Note that $\mathcal{C}_{P}$ is also nuclear as it is an ideal of the nuclear algebra $\T(P)$, and hence,
$A\otimes_{\textrm{max}} \mathcal{C}_{P}=A\otimes_{\textrm{min}} \mathcal{C}_{P}=A\otimes \mathcal{C}_{P}$. Consequently, for the system $(A,P,\id)$, we have the following
commutative diagram:
\begin{equation*}
\begin{diagram}\dgARROWLENGTH=0.4\dgARROWLENGTH
\node{0} \arrow{e}\node{\I} \arrow{s,l}{\Upsilon}\arrow{e}
\arrow{s}\arrow{e}\node{(\T_{\textrm{cov}}(A\times_{\id} P),i)}\arrow{s,l}{\Upsilon}\arrow{e,t}{\Omega}
\node{(A\times_{\id}^{\iso} P,\mu)}\arrow{s,l}{h}\arrow{e} \node{0}\\
\node{0} \arrow{e} \node{A\otimes \mathcal{C}_{P}} \arrow{e} \node {A\otimes \T(P)} \arrow{e,t}{\id\otimes \psi} \node {A\otimes C(\widehat{G})}
\arrow{e} \node{0,}
\end{diagram}
\end{equation*}
where $h$ is an isomorphism obtained by the composition of the isomorphisms
$$A\times_{\id}^{\iso} P\simeq A\rtimes_{\id} G\simeq A\otimes_{\max} C^{*}(G)\simeq A\otimes C(\widehat{G})$$ such that
$h(\mu_{A}(a)\mu_{P}(x))=a\otimes \varepsilon_{x^{-1}}$ for all $a\in A$ and $x\in P$.
\end{proof}

\subsection*{Acknowledgements}
The author would like to thank the anonymous reviewer for valuable suggestions and many helpful comments on earlier versions of the paper.

\end{document}